\documentclass[12pt,reqno]{amsart}

\author{Thomas John Baird and Michael Lennox Wong}
\title{E-polynomials of character varieties for real curves}

\usepackage{amsmath,amsthm,amsfonts,amssymb,bbm,amscd,mathrsfs}
\usepackage[all]{xy}
\usepackage{sseq}
\usepackage{comment}
\usepackage{enumerate}
\usepackage{bm}
\usepackage[
            breaklinks=true,
            pdftitle={},
            pdfauthor={Thomas John Baird, Michael Lennox Wong}]{hyperref}
\usepackage{color}
\definecolor{tocolor}{rgb}{.1,.1,.5}
\definecolor{urlcolor}{rgb}{.2,.2,.6}
\definecolor{linkcolor}{rgb}{.1,.4,.6}
\definecolor{citecolor}{rgb}{.6,.3,.1}
\hypersetup{colorlinks=true, urlcolor=urlcolor, linkcolor=linkcolor, citecolor=citecolor}

\newcommand{\E}{\mathbb{E}}
\newcommand{\F}{\mathbb{F}}

\newcommand{\Z}{\mathbb{Z}}

\newcommand{\C}{\mathbb{C}}
\newcommand{\Q}{\mathbb{Q}}
\newcommand{\PP}{\mathscr{P}}
\newcommand{\Nm}{\textnormal{Nm}}
\newcommand{\Tr}{\textnormal{Tr}}

\newcommand{\GL}{\textnormal{GL}}

\newcommand{\bmu}{\boldsymbol{\mu}}

\newcommand{\brho}{\boldsymbol{\rho}}
\newcommand{\norm}[1]{\left\lVert#1\right\rVert}
\newcommand{\s}{\textnormal{s}}

\newcommand{\odd}{\textnormal{odd}}
\newcommand{\hide}[1]{}

\newcommand{\sa}{\mathsf{a}}
\newcommand{\ssb}{\mathsf{b}}

\DeclareMathOperator{\Irr}{\textnormal{Irr}}

\DeclareMathOperator{\Hom}{\textnormal{Hom}}
\DeclareMathOperator{\supp}{\textnormal{supp}}

\newtheorem{thm}{Theorem}[section]
\newtheorem{cor}[thm]{Corollary}
\newtheorem{lem}[thm]{Lemma}
\newtheorem{prop}[thm]{Proposition}

\newtheorem{rmk}[thm]{Remark}

\theoremstyle{definition}
\newtheorem{eg}{Example}[section]

\numberwithin{equation}{section}

\begin{document}

\begin{abstract}
We calculate the E-polynomial for a class of (complex) character varieties $\mathcal{M}_n^{\tau}$ associated to a genus $g$ Riemann surface $\Sigma$ equipped with an orientation reversing involution $\tau$. Our formula expresses the generating function $\sum_{n=1}^{\infty}  E(\mathcal{M}_n^{\tau}) T^n$ as the plethystic logarithm of a product of sums indexed by Young diagrams. The proof uses point counting over finite fields, emulating Hausel and Rodriguez-Villegas \cite{HRV}.  
\end{abstract}

\maketitle


\section{Introduction}

Let $\Sigma$ be a compact Riemann surface of genus $g$ and let $ \Sigma^{''} := \Sigma \setminus \{p,p'\}$ be that same surface with two points removed. Given a field $\F$ and a primitive $2n$th root of unity $\xi \in \F$ consider the representation variety $$\mathcal{R}_n(\F)  := \Hom_{\xi}( \pi_1(\Sigma''), GL_n(\F))$$ of homomorphisms from the fundamental group $\pi_1(\Sigma'')$ to $GL_n(\F)$ which sends the positively oriented loops around both $p$ and $p'$ to $\xi I_n$, where $I_n$ is the identity matrix. When $\F = \C$, define the character variety 
\begin{equation}\label{charvarfirst}
 \mathcal{M}_n :=  \mathcal{R}_n(\C)/\!/ GL_n(\C)
 \end{equation}
to be the GIT quotient of $\mathcal{R}_n(\C)$ under the natural conjugation action by $GL_n(\C)$. This $\mathcal{M}_n$ is a smooth, affine, complex symplectic variety.  If  $\xi = e^{2\pi i d/2n}$, then $\mathcal{M}_n$ is naturally diffeomorphic to the moduli space of Higgs bundles of degree $d$ and rank $n$ over $\Sigma$ via the non-Abelian Hodge correspondence \cite{Hit, S}. 

If $\Sigma$ is endowed with an anti-holomorphic involution $\tau$ that interchanges $p$ and $p'$, then there is an associated character variety $\mathcal{M}_n^{\tau}$ introduced by Baraglia and Schaposnik \cite{BS} and by Biswas, Garc\'ia-Prada, and Hurtubise \cite{BGP}, which embeds as a holomorphic Lagrangian submanifold in $\mathcal{M}_n$ (also known as an ABA-brane \cite{BS}). If $\Sigma^\tau \neq \emptyset$, which we will usually assume, then $\mathcal{M}_n^\tau$ is equal to the the fixed point set of an anti-symplectic involution of $\mathcal{M}_n$. Under the non-Abelian Hodge correspondence, $\mathcal{M}_n^{\tau}$ is sent to the set of real points in the moduli space of Higgs bundles over the real curve $(\Sigma, \tau)$. 

In the present paper, we calculate the E-polynomial of $\mathcal{M}_n^{\tau}$  for all $n \geq 1$.  Our calculation reduces to counting points over finite fields, emulating the calculation of the E-polynomial of $\mathcal{M}_n$ by Hausel and Rodriguez-Villegas \cite{HRV}. 

The character variety $\mathcal{M}_n^{\tau}$ is defined using the orbifold fundamental group $ \widetilde{\pi_1(\Sigma'')}$,  which is the fundamental group of the homotopy quotient of $\Sigma''$ with respect to the $\Z/2$ action generated by $\tau$. This fits into a short exact sequence $ 1 \rightarrow \pi_1(\Sigma'') \rightarrow  \widetilde{\pi_1(\Sigma'')}   \rightarrow \Z/2 \rightarrow 0$ which splits if $\Sigma^\tau \neq \emptyset$.
Let $\widetilde{GL}_n(\F) = GL_n(\F) \rtimes \Z/2$ be the semi-direct product determined by the Cartan involution $A \mapsto (A^T)^{-1}$ and let $\mathcal{R}_n^\tau(\F)$ denote the representation variety of   homomorphisms $\tilde{\phi}$ that extend to a commutative diagram
$$ \xymatrix{  1  \ar[r] \ar[d] & \pi_1(\Sigma'') \ar[r] \ar[d]^{\phi} &  \widetilde{\pi_1(\Sigma'')}  \ar[r] \ar[d]^{\tilde{\phi}}  & \Z/2  \ar[r] \ar[d]^=  & 0 \ \ar[d]  \\ 1  \ar[r] &  GL_n(\F)  \ar[r] & \widetilde{GL}_n(\F)  \ar[r]  & \Z/2  \ar[r] & 0   .} $$ 
Define 
\begin{equation}\label{charvarrealfirst}
 \mathcal{M}_n^{\tau} := \mathcal{R}_n^\tau(\C) /\!/ GL_n(\C)
 \end{equation}
the GIT quotient under conjugation by $GL_n(\C) \leq \widetilde{GL}_n(\C)$. The forgetful map $\mathcal{M}_n^\tau \rightarrow \mathcal{M}_n$ is an embedding.

One of the main results of \cite{HRV} is a formula for the E-polynomial (or Serre characteristic) of $\mathcal{M}_n$. They prove the following remarkable generating function identity
\begin{equation}\label{HRVformula}
\frac{1}{(q-1)^2}\sum_{n=1}^{\infty}  \frac{ E(\mathcal{M}_n;q)}{ q^{n^2(g-1)}} T^n =    \mathrm{Log}\Big( \sum_{\lambda \in \PP} \mathcal{H}_{\lambda}^{2g-2}(q) T^{|\lambda|}\Big).
\end{equation}
In this expression $E(\mathcal{M}_n;q)$ is a polynomial in $q$ from which the E-polynomial of $\mathcal{M}_n$ is recovered by setting $q=xy$; the function $\mathrm{Log}$ is the plethystic logarithm; and $\mathcal{H}_{\lambda}(q)$ is the normalized hook polynomial associated to a partition (or Young diagram) $\lambda \in \PP$ (see \S \ref{E-polysec} for details). Our main result is an analogue of this formula for the E-polynomial of $\mathcal{M}_n^\tau$.  

\begin{thm}\label{BigThm}
Let $(\Sigma, \tau)$ be a genus $g$ Riemann surface equipped with an anti-holomorphic involution such that $\Sigma^\tau$ has $r$-many path components, with $r\geq 1$. \footnote{ The case $\Sigma^\tau = \emptyset$ has been considered by Letellier and Rodriguez-Villegas \cite{LR}. They derived a formula for the E-polynomial of $\mathcal{M}_n^{\tau}$ in that case, and produced a conjectural formula for the mixed Hodge polynomial which was recently disproven by Scognamiglio \cite{Sc}. Since we focus on the case $\Sigma^\tau \neq \emptyset$, their work complements ours very nicely.}
Then
\begin{equation}\label{GenFunctBig}
\frac{2}{q-1}  \sum_{n=1}^{\infty}  \frac{E(\mathcal{M}_n^{\tau};q)}{(-q^{\frac{1}{2}})^{n^2(g-1)}} T^n = \mathrm{Log}  \prod_{k=0}^{\infty} \left( \frac{ \sum_{\lambda \in \PP} (\sa^+_{\lambda})^r \mathcal{H}_{\lambda'}^{g-1}(q^{2^k}) T^{2^k|\lambda|}}{ \sum_{\lambda \in \PP} (\sa^-_{\lambda})^r \mathcal{H}_{\lambda'}^{g-1}(q^{2^k}) T^{2^k|\lambda|}}\right)^{ \frac{1}{2^k}}.
\end{equation}
where if $\lambda = (1^{s_1} 2^{s_2} ...)$ then 
\begin{align}\label{a+lam}
\sa^+_{\lambda}&= (s_1+1)(s_2+1)...\\\label{a-lam}
\sa^-_{\lambda}& = \begin{cases} 1 & \text{ if the conjugate partition $\lambda'$ has only even parts} \\ 0 & \text{otherwise.} \end{cases}
\end{align}
\end{thm} 
Furthermore, the character variety decomposes into connected components indexed by certain invariants $w$ 
$$ \mathcal{M}_n^\tau =  \coprod_w  \mathcal{M}_{n,w}^\tau $$
(see Corollary \ref{compscon}) and we calculate the E-polynomials of these components. 

The coefficients $a_\lambda^{\pm}$ appearing above are calculated using Schur functions $s_{\lambda}$. We prove identities
\begin{align*} 
\sum_{\lambda \in \PP} \sa^+_{\lambda}s_\lambda & =  \left( \sum_{\lambda \in \PP} s_\lambda \right) \left(\sum_{n=0}^{\infty}s_{1^n}\right) \\
\sum_{\lambda \in \PP} \sa^-_{\lambda'}s_\lambda & =  \left( \sum_{\lambda \in \PP} s_\lambda \right) /\left(\sum_{n=0}^{\infty} s_{1^n}\right)
\end{align*} 
from which the formulas (\ref{a+lam}), (\ref{a-lam}) are deduced using the Pieri rule.

\subsection{Outline of the proof of Theorem \ref{BigThm}}

Our proof of Theorem \ref{BigThm}, like the proof of (\ref{HRVformula}), relies on counting points over finite fields. 
Hausel and Rodriguez-Villegas \cite{HRV}, construct a polynomial $p(q) \in \Z[q]$ such that $| \mathcal{R}_n(\F_q)| =p(q)$ for $char(q) \gg 1$. By Katz' Theorem \cite{HRV} (alternatively Ito \cite[Cor. 6.5]{I}), this determines the E-polynomial by the identity
$$  E(\mathcal{R}_n(\C)) = p(xy) .$$
Furthermore, since $GL_n(\C)$ acts with constant stabilizer $\C^{\times} I_n$ so that $PGL_n(\C) := GL_n(\C)/ \C^{\times} I_n$ acts freely, we have
$$ E( \mathcal{M}_n) =  p(xy)  / E( PGL_n(\C)).$$
In the present paper, we prove that $ p^{\tau}(q)  =  | \mathcal{R}_n^{\tau}(\F_q)|$ is a polynomial function of $q$ for $char(q) \gg 1$ and deduce similarly
$$ E( \mathcal{M}_n^{\tau}) = p^{\tau}(xy) / E( GL_n(\C))$$
with the difference that $GL_n(\C)$ acts on $\mathcal{R}_n^{\tau}(\C)$ with constant stabilizer group $\pm I_n$ rather than $\C^{\times} I_n$. 

To get their point count formula, Hausel and Rodriguez-Villegas use the presentation\footnote{In fact Hausel and Rodriguez-Villegas work with a single puncture, but the resulting formula for $\mathcal{R}_n(\F)$ is identical.}
$$ \pi_1( \Sigma'') = \langle a_1,b_1,..., a_g, b_g, c, d| \prod_{i=1}^g [a_i, b_i] = cd \rangle,$$
which determines an isomorphism
$$ \mathcal{R}_n(\F) \cong  \{ (A_1,B_1,...,A_g,B_g) \in GL_n(\F)^{2g} | \prod_{i=1}^g [A_i, B_i] = \xi^2 Id_n\}.$$
Define the class function $ C:  GL_n(\F_q) \rightarrow \Z_{\geq 0} \subseteq \C$ by
$$ C(A) :=  | \{ (X,Y) \in GL_n(\F_q)^2| [X,Y] = A\}| .$$  
Then $$ | \mathcal{R}_n(\F_q)| = ( C * ... * C) (\xi^2 Id_n)  = C^{ * g} (\xi^2 Id_n) $$
where $*$ is the convolution product 
$$ (\phi * \psi) (A) = \sum_{B \in GL_n(\F_q)} \phi(B) \psi(B^{-1}A).$$

In similar fashion, we use an explicit presentation for $\pi_1( \Sigma'',\tau)$ due to Huisman \cite{H}  to derive the identity
\begin{equation}\label{convueq}
 |\mathcal{R}_n^{\tau} (\F_q) |=  (F *...* F  * N *...* N) (\xi Id_n) = (F^{*r} * N^{*(g-r+1)} )(\xi Id_n)
 \end{equation}
where $r$ is the number of path components of $\Sigma^{\tau}  \cong \coprod_{r} S^1$. In this expression, $F,N: GL_n(\F_q) \rightarrow \Z_{\geq 0} \subseteq \C$ are class functions defined by
$$ F(A): =  | \{ B \in GL_n(\F_q) |  B = B^T,  A B A^T = B \}|$$ 
$$N(A) :=  |\{ B \in GL_n(\F_q) |  B (B^{-1})^T = A \}| .$$
The function $N$ was considered by Gow \cite{G} where he proved that $$ N * N = C,$$ so $N$ is a ``square root" of $C$. The function $F(A)$ counts the number of non-degenerate symmetric bilinear form on $\F_q^n$ for which $A$ is an isometry. 

Convolution products can be understood using harmonic analysis. Recall that the irreducible characters of a finite group $G$ form an orthonormal basis of the space of class functions, so given a class function $\phi: G \rightarrow \C$ we have, 
\begin{equation}\label{orthonormalbasisdec}
 \phi = \sum_{\chi \in Irr(G)} \langle \phi, \chi\rangle   \chi
 \end{equation}
where 
\begin{equation}\label{stinprodcha}
\langle \phi , \chi\rangle  = \frac{1}{|G|} \sum_{g \in G} \phi(g) \overline{ \chi(g)}.
\end{equation} 
The convolution product satisfies the identity
\begin{equation}\label{FouTrans1}
  \langle \phi * \psi, \chi\rangle   =  \frac{|G|}{\chi(1) }  \langle\phi, \chi\rangle \langle\psi, \chi\rangle,
  \end{equation}
so convolution products are easily understood once class functions are expressed in the irreducible character basis. Gow \cite{G} proved that
$$ N  =  \sum_{\chi \in \Irr (GL_n(\F_q))}  \chi$$
so every irreducible character occurs with multiplicity one.\footnote{The multiplicity in this case can be interpreted as a twisted Frobenius-Schur indicator (see \cite{KM}), which helps explain why it always equals one.} 
Define $\sa_{\chi}\in \Z_{\geq 0}$ by
\begin{equation}\label{achidec}
 F = \sum_{\chi \in \Irr(GL_n(\F_q))}  \sa_{\chi} \chi .
 \end{equation}
Applying to (\ref{convueq}) we get
\begin{eqnarray*}
 |\mathcal{R}_n^{\tau} (\F_q) | &=& |GL_n(\F_q)|^{g}   \sum_{\chi \in \Irr (GL_n(\F_q))}  \sa_{\chi}^r   \frac{ \chi(\xi I_n )}{\chi(I_n)^g}.
\end{eqnarray*} 
The bulk of the current paper is devoted to calculating the coefficients $\sa_{\chi}$. The calculation breaks into three steps:

\begin{enumerate}
\item In \S \ref{Explicitform}, we compute an explicit formula for $F$ using Milnor's classification of orthogonal transformations over perfect fields \cite{Mi}.

\item In \S \ref{FouTrans}, \S \ref{Determining multiplicities}, we use the formula $\sa_{\chi}= \langle F, \chi \rangle$ (see (\ref{orthonormalbasisdec})) to calculate $\sa_{\chi}$ in the limit as $q \rightarrow \infty$. 

\item In \S \ref{Characters of}, we show that the limit in step 2 actually makes sense, and that the formula remains valid for $char(q) \gg 1$.
\end{enumerate}

Once the multiplicities $\sa_\chi$ have been determined it is relatively straightforward (emulating \cite{HRV}) to produce an explicit polynomial expression for (\ref{convueq}), leading to the generating function in Theorem \ref{BigThm}.  This is carried out in \S \ref{E-polysec}.  Finally in \S \ref{ccm}, we calculate the E-polynomial of the connected components $\mathcal{M}_{n,w}^\tau$.

\subsection{Further discussion}

\subsubsection{Euler characteristic of the $PGL_n$-character variety}\label{Aref}

The character variety $\mathcal{M}_n$ admits an action by the group $\mathcal{A} = Hom( \pi_1(\Sigma),\C^{\times}) \cong (\C^{\times})^{2g}$ via scalar multiplication. The quotient space $ \tilde{\mathcal{M}}_n :=  \mathcal{M}_n/\!/ \mathcal{A}$ is identified with the $PGL_n(\C)$-character variety. If $g\leq 1$ then $\tilde{\mathcal{M}}_n$ is either a point or the empty set. If $g \geq 2$, Hausel and Rodriguez-Villegas proved that the Euler characteristic of $\tilde{\mathcal{M}}_n$ is equal to $\mu(n) n^{2g-3}$ where $\mu(n)$ is the Moebius function. 

The involution $\tau$ on $\Sigma$ induces an automorphism of $\mathcal{A}$ and the invariant subgroup $ \mathcal{A}^\tau \cong (\C^{\times})^g \times \{\pm 1\}^{r-1}$ acts on $\mathcal{M}_n^\tau$. Define  
$$ \tilde{\mathcal{M}}_n^\tau :=  \mathcal{M}_n^\tau/\!/ \mathcal{A}^\tau.$$
If $n$ is odd then $\mathcal{A}$ transitively permutes the set of  connected components of $\mathcal{M}_n^\tau$. Consequently if $\mathcal{M}_{n,w}^\tau$ is a particular component, we have 
$$ \tilde{\mathcal{M}}_n^\tau  \cong  \mathcal{M}_{n,w}^\tau/\!/ \mathcal{A}^\tau_0$$
where $\mathcal{A}_0 \cong (\C^{\times})^g$ be the identity component of $\mathcal{A}$.
On the other hand, if $n$ is even, then $\mathcal{A}$ does not permute components and we have
$$  \tilde{\mathcal{M}}_n^\tau  \cong  \coprod_w  \mathcal{M}_{n,w}^\tau/\!/ \mathcal{A}^\tau .$$

When $g \geq 2$  and $n$ is odd, we prove (Corollary \ref{euler char}) that the Euler characteristic of $\tilde{\mathcal{M}}_n^\tau$ is equals to $\mu(n) n^{g-2}$. When $g \geq 2$ and $n$ is even, we prove that the Euler characteristic of $ \mathcal{M}_{n,w}^\tau/\!/ \mathcal{A}^\tau_0$ is zero for all $w$, but we do not calculate the Euler characteristic of $\tilde{\mathcal{M}}_n^\tau$.

\subsubsection{Mixed Hodge polynomials}

The E-polynomial is a specialization of the (compactly supported) mixed Hodge polynomial.  Namely, if $Z$ is a complex variety,  $E(Z;x,y) := MH( Z; x,y,-1)$ where 
$$MH(Z;x,y,t):=  \sum  h_{i,j,k} x^i y^j t^k$$
and $h_{i,j,k}$ are the dimensions of associated graded components of the mixed Hodge filtration on compactly supported cohomology $H_c^k( Z;\C)$.  Hausel and Rodriguez-Villegas conjectured a generating function identity for mixed Hodge polynomial which reduces to (\ref{HRVformula}) upon setting $t=-1$:
\begin{equation}\label{HRVformulaMH}
\frac{1}{(xy-1)(t^2xy-1)}\sum_{n=1}^{\infty}  \frac{ MH(\mathcal{M}_n;x,y,t)}{(t^2xy)^{n^2(g-1)}} T^n =   \mathrm{Log}\left( \sum_{\lambda \in \PP} \mathcal{H}_{\lambda}^{2g-2}(xy,t) T^{|\lambda|}\right).
\end{equation}
for certain two variable rational functions  $\mathcal{H}_{\lambda}^{2g-2}(xy,t)$.  They verified that (\ref{HRVformulaMH}) gives the correct formula for $n=1$ and $n=2$.  The specialization $x=y=1$ was later proven by Mellit \cite{Me}.

It is natural to hope for an analogous conjectural identity for the mixed Hodge polynomial of $\mathcal{M}_n^\tau$. Our efforts in this direction have been hampered by a lack of understanding of the rational cohomology ring of $\mathcal{M}_n^\tau$ when $n>1$; so far only the $\Z_2$-Betti numbers of $\mathcal{M}_2^\tau$ have been calculated \cite{B}.  This is a promising direction for future research.

\subsubsection{Curious Poincar\'e duality}
The E-polynomial of $\mathcal{M}_n$ satisfies the so-called curious Poincar\'e duality property
\begin{equation}\label{cpd}
  E(\mathcal{M}_n; q) =  q^{\dim(\mathcal{M}_n)} E(\mathcal{M}_n;q^{-1}).
\end{equation}
This is curious because $\mathcal{M}_n$ is non-compact, so topological Poincar\'e duality does not apply. Curious Poincar\'e duality is a consequence of the P=W conjecture of de Cataldo, Hausel, and Migliorini \cite{CHM}, which holds that under the non-Abelian Hodge correspondence homeomorphism between $\mathcal{M}_n$  and $\mathcal{M}_{Dol}$, the weight filtration on the cohomology ring $H^*(\mathcal{M}_n;\Q)$ is sent to the perverse Leray filtration on $H^*(\mathcal{M}_{Dol}; \Q)$ associated to the Hitchin map  $\mathcal{M}_{Dol} \rightarrow \mathcal{B}$. The $P=W$ conjecture was proven recently by Maulik and Shen \cite{MS}, and independently by Hausel, Mellit, Minets, and Schiffmann \cite{HMMS}.

In the current paper, we find that the E-polynomial of $\mathcal{M}_n^{\tau}$ generally  \emph{does not} satisfy curious Poincar\'e duality. This is in keeping with the P=W conjecture, because under the non-Abelian Hodge correspondence $\mathcal{M}_n^{\tau}$ is sent to a real submanifold $\mathcal{M}_{Dol}^{\tau} \subseteq \mathcal{M}_{Dol}$ and we should not expect the real integrable system $\mathcal{M}_{Dol}^{\tau} \rightarrow \mathcal{B}^{\tau}$ to have a well behaved perverse Leray spectral sequence \cite{CM}. As such, our results provided some circumstantial evidence in support of the P=W conjecture (we thank Vivek Shende for emphasizing this point to us).

\hide{
\subsubsection{The $q=1$ case}
Morally, the symmetric group $S_n$ can be thought of as the general linear group over the ``field of one element", $GL_n(\F_1)$. It seems appropriate then that the general formula for the coefficients $\sa_\chi$ in (\ref{achidec}) reduces to a pair of $S_n$ characters
\begin{align*}
 \chi_+ & := \sum_{\lambda \vdash n} \sa_{\lambda'}^+ \chi_{\lambda} &  \chi_- &:=   \sum_{\lambda \vdash n} \sa_{\lambda'}^- \chi_{\lambda} 
 \end{align*}
where $\chi_\lambda$ is the irreducible character of $S_n$ associated to the partition $\lambda$.  Since $F$ is the permutation character for the natural action of $GL_n(\F_q)$ on the set of non-degenerate symmetric bilinear forms, one might expect that $\chi_+$ and $\chi_-$ are permutation characters for the action of $S_n$ on some set analogous to the set of non-degenerate symmetric bilinear forms over $\F_1$.  

It is not hard to show that $\chi_-$ is the character for the permutation representation of $S_{n}$ on the set of perfect matching graphs on $n$ vertices (so $\chi_- = 0$ if $n$ is odd). We have verified that $\chi_+$ is a permutation character for $n \leq 7$, but we have not yet found a nice geometric interpretation of $\chi_+$. This remains an intriguing open problem. }

\textbf{Acknowledgements:} This research collaboration began during consecutive conferences in ICTS Bangalore and TIFR Mumbai and we thank Indranil Biswas for inviting us both. The first author was supported by an NSERC Discovery Grant RGPIN-2016-05382 and the second by the Collaborative Research Center SFB/TR 45 `Periods, moduli spaces and arithmetic of algebraic varieties' (Project M08-10) of the Deutsche Forschungsgemeinschaft. Thanks also to Michael Groechenig, Tam\'as Hausel, Anton Mellit, and the referee for helpful comments on an earlier draft or presentation.

\textbf{Declarations}  On behalf of all authors, the corresponding author states that there is no conflict of interest. Data sharing not applicable to this article as no datasets were generated or analysed during the current study.

\section{The character variety}\label{the character variety}

We begin with a review of Hitchin's equations and the non-abelian Hodge correspondence. Let $P$ be a $U(n)$-principle bundle of degree $d$ over a compact Riemann surface $\Sigma$ and where $d, n$ are coprime, and let $P^c$ be the complexified $GL_n(\C)$-bundle. Let $A$ be a connection on $P$ and let $\Phi \in \Omega^{(1,0)}(\Sigma, ad P^c)$.  The (inhomogeneous) Hitchin equations are
\begin{eqnarray*}
F_A  + [ \Phi, \Phi^*] &=& \omega \\
d^{''}_A \Phi & =& 0
\end{eqnarray*}
where $\omega \in \Omega^2( \Sigma, z) $ is a fixed 2-form with values in the centre $z \subseteq ad P^c$. The moduli space of solutions is a manifold we denote $\mathcal{M}_{Hit}(\omega)$. Given any two $\omega, \omega' \in \Omega^2( \Sigma, z) $ we can produce an isomorphism $\mathcal{M}_{Hit}(\omega) \cong \mathcal{M}_{Hit}(\omega')$ by tensoring with an appropriate $U(1)$-bundle connection, so we will abuse notation and denote $\mathcal{M}_{Hit}:= \mathcal{M}_{Hit}(\omega)$.

The forgetful map $(A, \Phi) \mapsto (d_A^{''}, \Phi)$ induces a morphism from $\mathcal{M}_{Hit}$ to the moduli space of stable Higgs pairs $\mathcal{M}_{Dol}$. The forgetful map $(A,\Phi) \mapsto A+ \Phi + \Phi^*$ determines a morphism from $\mathcal{M}_{Hit}$ to the moduli space $\mathcal{M}_{DR}$ of $GL_n(\C)$-connections with curvature $\omega$. Note in particular that $A+ \Phi + \Phi^*$ is projectively flat. The non-Abelian Hodge correspondence says that the forgetful maps defined above determine diffeomorphisms
\begin{equation}\label{nahc}
 \mathcal{M}_{Dol} \leftarrow  \mathcal{M}_{Hit}  \rightarrow \mathcal{M}_{DR}.
\end{equation}

Now consider an anti-holomorphic involution $\tau: \Sigma \rightarrow \Sigma$ and suppose  $\omega =\tau^*\overline{\omega} $. This determines an involution on $\mathcal{M}_{Hit}$ sending the pair $(A, \Phi)$ on $P$ to the pair $\tau(A,\Phi)  =  ( \tau^*\overline{A}, -\tau^* \overline{\Phi})$ on the conjugate pull-back bundle $\tau^* \overline{P}$. The involution descends to a holomorphic involution of $\mathcal{M}_{DR}$ and an anti-holomorphic involution of $\mathcal{M}_{Dol}$. By Propostion \ref{nahc} we obtain diffeomorphisms of fixed point sets
$$ (\mathcal{M}_{Dol})^{\tau} \cong  (\mathcal{M}_{Hit})^{\tau}  \cong (\mathcal{M}_{DR})^{\tau}.$$

Choose $\omega$ equal to zero except for a pair of delta function singularities at a pair of  points $p, p'$ and let $\Sigma'' := \Sigma \setminus \{p,p'\}$. Then the Riemann-Hilbert correspondence determines a diffeomorphism to the character variety (\ref{charvarfirst}) 
$$ \mathcal{M}_{DR}  \cong \mathcal{M}_n := \mathcal{R}_n(\C)  /\!/  GL_n(\C). $$

Assume henceforth that $\Sigma^{\tau} \neq \emptyset$ and choose one of these fixed points as the base point for $\pi_1(\Sigma'')$ so that $\tau$ induces an automorphism $\tau_*$ of $\pi_1(\Sigma'')$.  The following is a minor alteration of (\cite{BS}, Prop. 15).

\begin{prop}
If $ \Sigma^\tau \neq \emptyset$ then the forgetful map $\mathcal{M}_n^{\tau} \rightarrow \mathcal{M}_n$ restricts to a bijection
$$ \mathcal{M}_n^{\tau} \cong ( \mathcal{M}_n)^{\tau}. $$
\end{prop}  

\begin{proof}
The isomorphism $\widetilde{\pi_1(\Sigma'')} = \pi_1(\Sigma^{''}) \rtimes \Z_2$ determines an isomorphism
\begin{eqnarray}
 \mathcal{R}_n^\tau(\C) &  \cong  &  \{ (\rho, A) \in \mathcal{R}_n(\C)  \times GL_n(\C) | A \theta( \rho)  \theta (A)  =  \rho \circ \tau_*, A \theta(A) = I_n\} \\
 & =&  \{ (\rho, A) \in \mathcal{R}_n(\C)  \times GL_n(\C) | A^{-1} \rho A =  \theta \circ \rho \circ \tau_*, A = A^T\} 
 \end{eqnarray}
where $\tau_*$ is the automorphism of $\pi_1(\Sigma'')$ induced by $\tau$ and $\theta \in Aut(GL_n(\C))$ is the Cartan involution $ \theta(X) =  (X^{-1})^T$. 

On the other hand, the involution $\tau$ on $\mathcal{M}_n$ lifts to the involution $\iota$ of $\mathcal{R}_n(\C)$ 
$$\iota(\rho) :=  \theta \circ \rho \circ \tau_*.$$ A homomorphism $ \rho \in  \mathcal{R}_n(\C)$ represents $[\rho] \in (\mathcal{M}_n)^{\tau}$ if and only if there exists $A \in GL_n(\C)$ such that 
\begin{equation}\label{Aeqn}
A^{-1} \rho A = \theta \circ \rho \circ \tau_* ,
\end{equation}
which implies, if $B :=A^TA^{-1}$, that
\begin{equation}\label{ttoinv}
B  \rho B^{-1} =  \rho.
\end{equation}
Since  $\rho$ is irreducible, $B$ is a scalar matrix, hence $B = \pm I_n$ and $A = \pm A^T$. We call $\rho$ \emph{real} if $A = A^T$ and \emph{quaternionic} if $A = -A^T$.  It remains to show there are no quaternionic representations.

Under the non-Abelian Hodge correspondence, the real representations are sent to real vector bundles and the quaternionic representations are sent to quaternionic vector bundles in the sense of Atiyah \cite{A}.  Therefore by (\cite{BHH} Prop. 4.2), the quaternionic representations do not exist if $\Sigma^{\tau} \neq \emptyset$. 

\end{proof}

A presentation of $\widetilde{\pi_1(\Sigma'')} $ was produced by Huisman \cite{H}. \footnote{In fact Huisman considered the case with no punctures, but the formula for $\widetilde{\pi_1(\Sigma'')} $ is an immediate corollary.}

\begin{prop}
Let $\Sigma''$ be a twice punctured genus $g$ Riemann surface with orientation reversing involution $\tau$ transposing the punctures. Let $r$ be the number of fixed point components of $\Sigma^{\tau}$, and let $r + s = g+1$. 
Then 
$$\widetilde{\pi_1(\Sigma'')} \cong \langle \{a_i, b_i\}_{i=1}^r, \{ x_j\}_{j=1}^s, d  | b_i^2=1, a_ib_i = b_ia_i, \Phi(a,x) =d \rangle $$
where
$$\Phi = \begin{cases} \prod_{i=1}^r a_i\prod_{j=1}^{s/2} [x_{2j-1},x_{2j}] & \text{ if  $\Sigma/\tau$ is orientable ($s$ is even in this case)}\\
 \prod_{i=1}^r a_i\prod_{j=1}^{s} x_{j}^2 & \text{ if not} .\end{cases} $$

The generators  $a_i$ and $d$ lie in the subgroup $\pi_1(\Sigma)$,  the $b_i$ do not, and the $x_j$ do if and only if $\Sigma/\tau$ is orientable. The generator $d$ corresponds to a loop around one puncture point.  
\end{prop}

It follows that $\mathcal{R}_n^\tau(\F)$ can be identified with subvariety of  $GL_n(\F)^{2r+s}$ of tuples  $(A_i, B_i,X_j)$ defined by equations
\begin{equation}\label{EqforF}
B_i = B_i^T,  \text{ and }  A_i B_i A_i^T =B_i, \text{ for all $i \in \{1,...,r\}$}
\end{equation}
and
$$\Phi(A,X) = \xi Id_n $$
where
\begin{equation}\label{Eqforprod}
\Phi := \begin{cases} \prod_{i=1}^r A_i\prod_{k=1}^{s/2} [X_{2k-1},X_{2k}] & \text{if   $\Sigma/\tau$ is orientable}.\\
\prod_{i=1}^r A_i \prod_{j=1}^s X_j(X_j^T)^{-1} & \text{if not}  \end{cases}
\end{equation}

Note that (\ref{EqforF}) simply requires $B_i$ to represent a non-degenerate, symmetric bilinear form with respect to which $A_i$ is orthogonal. This implies in particular that $\det(A_i) \in \{ \pm 1\}$ for all $i=1,...,r$, immediately yielding the following.

\begin{cor}\label{compscon}
There are coproduct decompositions
\begin{align*}
\mathcal{R}_n^\tau(\F) & = \coprod_{w}  \mathcal{R}_n^\tau(\F)_w & & \text{and} & \mathcal{M}_n^\tau & = \coprod_{w}  \mathcal{M}_{n,w}^\tau
\end{align*}
indexed by r-tuples $w \in \{\pm 1\}^r$ satisfying the condition $ \prod_{i=1}^r w(i) = \xi^n = -1$.
\end{cor}
The $\mathcal{M}_{n,w}^\tau$ are in fact the connected components of $\mathcal{M}_{n}^\tau$. This can be proven of the homeomorphic space $(\mathcal{M}_{Dol})^\tau$ using Morse theory (see for example \cite{B}). This is also clear from the E-polynomial formulas we calculate in \S \ref{ccm}.

The group $\mathcal{A}:= Hom(\pi_1(\Sigma), \C^{\times}) \cong (\C^{\times})^{2g}$ acts on the character variety $\mathcal{M}_n$ via the scalar multiplication action of $\C^{\times}$ on $GL_n(\C)$.  The quotient 
$$ \tilde{\mathcal{M}}_n :=  \mathcal{M}_n /\!/  \mathcal{A}$$
is called the $PGL_n(\C)$-character variety. 

The subgroup $ \mathcal{A}^\tau:= \{ \phi \in \mathcal{A} |  \theta \circ \phi \circ \tau_* = \phi\}$ restricts to an action on $\mathcal{M}_{n}^\tau$. Denote the quotient
$$ \tilde{\mathcal{M}}_n^\tau :=  \mathcal{M}_n^\tau /\!/  \mathcal{A}^\tau.$$
We have isomorphisms
$$ \mathcal{A}^\tau \cong  (\C^\times)^g \times \{ \pm 1\}^{r-1}$$
In terms of the presentation above, the action is defined
$$ (\lambda_1,..., \lambda_{g+1}, \epsilon_1,...,\epsilon_r) \cdot  (A_i, B_i, X_j) = (\epsilon_i A_i, \lambda_iB_i, \lambda_{j+r} X_j) $$
where $\lambda_i \in \C^{\times}$, $\epsilon_i \in \{\pm 1\}$ and we impose $\lambda_{g+1}=1 = \prod_{i=1}^r \epsilon_i $. 
Under this action $\mathcal{M}_1^\tau$ is an $\mathcal{A}^\tau$-torsor, so

\begin{equation}\label{n=1eg}
\mathcal{M}_1^{\tau} \cong  \coprod_{2^{r-1}}  (\C^\times)^g.
\end{equation}

\begin{rmk}\label{simpacm}
When $n$ is odd, $\mathcal{A}^\tau$ transitively permutes the connected components $\mathcal{M}_{n}^\tau$, so the $\mathcal{M}_{n,w}^\tau$ are pair-wise isomorphic.  When $n$ is even $\mathcal{A}^\tau$ does not permute components.
\end{rmk}

The surjective homomorphism $\widetilde{GL}_n(\C) \rightarrow \widetilde{GL}_1(\C)$  defined by sending $(A, \epsilon) \rightarrow (\det(A), \epsilon)$ determines a fibre bundle
$$det: \mathcal{M}_{n,w}^\tau \rightarrow \mathcal{M}_{1,w}^\tau  \cong (\C^{\times})^g.$$

Given $ \phi \in \mathcal{M}_{1,w}^\tau$ denote the fibre $\mathcal{M}_{n,\phi}^\tau :=  \det^{-1}(\phi)$. If $\mathcal{A}^\tau_0 \cong (\C^{\times})^g$ is the identity component of $\mathcal{A}$ then we have an isomorphism
$$ \mathcal{M}_{1,w}^\tau \cong \mathcal{M}_{1,\phi}^\tau \times_{\mu_n} \mathcal{A}^\tau_0 $$ 
where $\mu_n \cong (\Z/n)^g$ is the $n$-torsion subgroup of $\mathcal{A}^\tau_0$. Therefore
$$ H^*(\mathcal{M}_{1,w}^\tau) \cong  H^*(\mathcal{M}_{1,\phi}^\tau)^{\mu_n} \otimes H^*((\C^{\times})^g),$$
where $H^*(\mathcal{M}_{1,\phi}^\tau)^{\mu_n}$ is the ring of $\mu_n$-invariants. In particular, the E-polynomial of $\mathcal{M}_{1,w}^\tau$ is divisible by $E((\C^{\times})^g) = (q-1)^g$.

\section{Point counting and the E-polynomial}

\begin{prop}\label{quotienteplo}
The conjugation action of $GL_n(\C)$ on $\mathcal{R}_n^\tau(\C)$ is free modulo the centre $\{\pm I_n\} \leq GL_n(\C)$. Consequently, the E-polynomials satisfy the identity
\begin{align*} 
 E( \mathcal{M}_{n}^\tau)   &= \frac{E( \mathcal{R}_n^\tau(\C) )}{ E( GL_n(\C))}.
 \end{align*}
Similarly
\begin{align*} 
 E( \mathcal{M}_{n,w}^\tau)   &= \frac{E( \mathcal{R}_n^\tau(\C)_w )}{ E( GL_n(\C))}.
 \end{align*}
\end{prop}

\begin{proof}
The forgetful map $\mathcal{R}_n^\tau(\C) \rightarrow \mathcal{R}_n(\C)$ is $GL_n(\C)$-equivariant and it was proven in \cite[Lemma 2.2.6]{HRV} that every point in $\mathcal{R}_n(\C)$ is stabilized only by scalar matrices. However, the only scalar matrices that centralize elements in the non-identity component of $\widetilde{GL}_n(\C)$ are $\{ \pm I_n\}$ so the quotient map $$  \mathcal{R}_n^\tau(\C) \rightarrow \mathcal{M}_n^\tau$$ is a principal $GL_n(\C)/ \{\pm I_n\}$-bundle. Since $GL_n(\C) \cong GL_n(\C)/ \{\pm I_n\}$ we have
$$  E( \mathcal{R}_n^\tau(\C) ) = E( \mathcal{M}_n^\tau) E( GL_n(\C))$$
by \cite[Remark 2.5]{LMN}. 
\end{proof}

It remains to calculate $ E( \mathcal{R}_n^\tau(\C) )$. We use a point counting argument analogous to that used by Hausel and Rodriguez-Villegas \cite{HRV}.

\begin{prop}\label{invokeKatz}
Suppose $p^\tau(t) \in \Z[t]$ is a polynomial such that the cardinality  $p^{\tau}(q) = |\mathcal{R}_n^\tau(\F_q)|$ for $char(q) \gg 1$. Then $p(xy) = E( \mathcal{R}_n^\tau(\C) )$. 
\end{prop}

\begin{proof}
Let $\Phi_{d}(x)$ the $d$th cyclotomic polynomial and consider the ring $A:= \Z[x]/(\Phi_{2n}(x))$. We can interpret $\mathcal{R}^{\tau}_n$ as an affine scheme over $A$ and $\mathcal{R}^{\tau}_n(\F)$ as the variety obtained by an extension of scalars $\phi: A \rightarrow \F$ which sends $x$ to the chosen primitive $2n$th root of unity $\xi \in \F$. The result now follows by Katz' Theorem \cite[Thm. 2.1.8]{HRV}.
\end{proof}

Next, we want an expression for $ |\mathcal{R}_n^\tau(\F_q)|$. Define functions $N, F, C$ from $G_n := GL_n(\F_q)$ to $\Z_{\geq 0}$ as follows:
\begin{align*}
 F(A) & := \left| \left\{ B \in G_n \, : \, B^T = B, ABA^T = B \right\} \right| \\
 N(A) & :=  \left| \left\{ B \in G_n \, : \, B (B^T)^{-1} = A \right\} \right| \\
 C(A) & := \left| \left \{ (X,Y) \in G_n^2 \, : \, [X,Y] = X Y X^{-1} Y^{-1} = A \right\} \right|.
\end{align*} 
If $\Sigma^\tau$ has $r$-path components, let $r+s = g+1$. It follows from (\ref{EqforF}) and (\ref{Eqforprod}) that the cardinality of $ \mathcal{R}_n^\tau(\F_q)$ is equal to the value of the following convolution product at $\xi I_n$:
$$ | \mathcal{R}_n^{\tau} (\F_q) | = \begin{cases}
\left( F^{*r} * C^{* s/2}  \right) (\xi Id_n) & \text{ if $\Sigma/\tau$ is orientable ($s$ is even in this case)}\\
  \left(F^{*r} * N^{*s} \right)(\xi Id_n) & \text{ if not} .\end{cases}$$
In fact, 
\begin{equation}\label{N2C}
 N*N = C,
 \end{equation}
so  
\begin{equation}\label{conprodfu}
| \mathcal{R}_n^\tau(\F_q)|  =  F^{*r} * N^{*(g-r+1)} (\xi I_n) .
\end{equation}
is independent of the orientability of $\Sigma/\tau$.  Gow \cite{G} proved that $N$ is the sum of the irreducible characters of $G_n$, each with multiplicity one $$ N = \sum_{\chi \in \Irr G_n} \chi.$$
The decomposition of $C$ can be found in (\cite{HRV} (2.3.7)) and the identity (\ref{N2C}) follows using (\ref{FouTrans1}). 

Observe that $F$ is a character function on $G_n$ because $F(g)$ counts $g$-fixed points for the action of $G_n$ on the set of non-degenerate symmetric bilinear forms on $\F_q^n$. This implies that 
$$F =  \sum_{\chi \in \Irr G_n} \sa_\chi \chi. $$
where the $\sa_\chi \in \Z_{\geq 0}$ are multiplicities of irreducible characters.

\begin{cor}\label{FormulaCor0}
If the function
\begin{eqnarray*}
E_n(q) &=& |G_n|^{g-1}  \sum_{\chi \in \Irr G_n}  \frac{\chi(\xi)}{\chi(1)^g} \sa_{\chi}^r\\
\end{eqnarray*}
is a polynomial function in $q$ for $char(q) \gg 1$, then 
\begin{eqnarray*}
 E(\mathcal{M}_n^\tau) &=& E_n(xy).
\end{eqnarray*}
\end{cor}

\begin{proof}
Applying (\ref{FouTrans1}) and (\ref{conprodfu}) we get
$$ | \mathcal{R}_n^\tau(\F_q)| = |G_n|^{g}  \sum_{\chi \in \Irr G_n}  \frac{\chi(\xi)}{\chi(1)^{g}} \sa_{\chi}^r = |G_n| E_n(q).$$
We have $|G_n| = f(q)$ where $f(t) = \prod_{i=0}^{n-1} (t^n -t^i)$, so if $E_n(q)$ is a polynomial function for $char(q) \gg1$, then by Proposition \ref{invokeKatz} we have
$$E(\mathcal{R}_n^\tau(\C)) = f(xy) E_n(xy).$$ 
Lastly, note that $E(GL_n(\C)) = f(xy)$ and apply Proposition \ref{quotienteplo}.
\end{proof}

Suppose now that $\text{char} \, \F \neq 2$. Define $$F = F_+ + F_-$$ where $F_+$ is supported on the matrices with determinant $1$ and $F_-$ is supported on those of determinant $-1$. Let $1 \leq k \leq r$ be odd and choose $w \in \{\pm 1\}^r$ for which $k$-many coordinates equal $-1$. Then the cardinality of $ \mathcal{R}_n^\tau(\F_q)_w$ is equal to 
$$|\mathcal{R}_n^\tau(\F_q)_w| =  F_+^{*(r-k)}  * F_-^{*k}* N^{*(g-r+1)} (\xi I_n). $$
If
 \begin{align*}
 F_+ &:= \sum_{\chi} \ssb_\chi^+ \chi&  F_- &:= \sum_{\chi}\ssb^-_\chi \chi, 
 \end{align*}
then similar reasoning yields
\begin{cor}\label{FormulaCor}
If the function
\begin{eqnarray*}
E_n^k(q) &:=& |G_n|^{g-1}  \sum_{\chi \in \Irr G}  \frac{\chi(\xi)}{\chi(1)^{g}} (\ssb^+_{\chi})^{r-k}(\ssb^-_{\chi})^{k}\\
\end{eqnarray*}
is a polynomial function of $q$ for $ char(q)\gg 1$, then 
\begin{eqnarray*}
 E(\mathcal{M}_{n,w}^\tau) &=& E_n^k(xy).
\end{eqnarray*}
\end{cor}

Let $\rho$ be the representation of $GL_n(\F_q)$ such that $F = tr(\rho)$, let  $\chi =  \iota \circ det$ the composition of the determinant map $GL_n(\F_q) \rightarrow \F_q^\times$ with an injective homomorphism $\iota: \F_q^{\times} \rightarrow \C^\times$. Then $\tilde{F} = tr(\rho \otimes \chi)$ is a character equal to $F_+ - F_-$.  Therefore, if $\tilde{F}:= \sum \tilde{\sa}_{\chi} \chi$, then
\begin{align}\label{Ftilde}
\ssb^+_\chi  & =  \frac{1}{2}(\sa_\chi+\tilde{\sa}_\chi) &  \ssb^-_\chi & =  \frac{1}{2}(\sa_\chi -\tilde{\sa}_\chi).
\end{align}

\section{Conjugacy classes of \texorpdfstring{$GL_n(\F_q)$}{Lg}}

This section establishes notation for finite fields, partitions, and conjugacy classes. We mostly follow \cite{M}.

\subsection{Fields}

Denote $\F_q$ the finite field of order $q$ with algebraic closure $\overline{\F}_q$. The Frobenius map $Frob : \overline{\F}_q \to \overline{\F}_q$ is given by $x \mapsto x^q$. For $n \geq 1$, we identify $\F_{q^n} \subset \overline{\F}_q$ with the fixed point set of $Frob^n$. Denote the multiplicative groups by
\begin{align*}
M_n & := \F_{q^n}^\times,  & M & := \bigcup_{n} M_n =  \overline{\F}_q^\times.
\end{align*}
Denote the orbit set $\Phi = M/Frob$, and $\Phi_d \subseteq \Phi$ the set of orbits of order $d$. Each orbit in $\Phi$ is equal to the set of roots of an monic irreducible polynomial and we represent elements $f \in \Phi$ by the corresponding polynomial. Write $ d = d_f$ for $f \in \Phi_d$. 

The automorphism $x \mapsto x^{-1}$ of $M$ determines the automorphism $f \mapsto f^*$ of $\Phi$. Define
\begin{align}\label{definephis}
\Phi^s &:= \{ f \in \Phi | f=f^*\} & \Phi^{p} &= \{ \{f, f^*\} | f \neq f^*\}.
\end{align} 
Observe that 
\begin{align} \label{e:Phi1s}
\Phi_1^s = \{ t-1, t+1 \};
\end{align}
and
\begin{align*}
\Phi_{>1}^s := \Phi^s \setminus \Phi_1^s =  \bigcup_{d>1} \Phi_{2d}^s.
\end{align*}

\subsection{Partitions} \label{s:partitions}

Let $$\PP = \bigcup_{n \geq 0} \PP_n$$ where $\PP_n$ is the set of partitions of $n$ (note $\PP_0 = \{ \emptyset\}$). If $\lambda = (\lambda_1 \geq \cdots \geq \lambda_\ell) $ is a partition, define
\begin{align*}
| \lambda | & := \sum_{i=1}^\ell \lambda_i, & \ell(\lambda) & := \ell, &  n(\lambda) := \sum_{i=1}^\ell (i-1) \lambda_i.\\
\end{align*}
Set $\ell_{\odd}(\lambda) + \ell_{ev}(\lambda) = \ell(\lambda)$ where $ \ell_{\odd}(\lambda) := | \{ i \, : \, \lambda_i \text{ is odd} \}|$ and  set $sgn(\lambda) = (-1)^{\ell_{ev}(\lambda)}$. For $d \in \Z_{>0}$, set
\begin{align*}
m_{d}(\lambda) := | \{ i \, | \, \lambda_i = d \}|,
\end{align*}
called the \emph{multiplicity of $d$ in $\lambda$}. 

We also use notation $\lambda = ( 1^{m_1} 2^{m_2} ....)$ for $m_d = m_{d}(\lambda)$.  If $s \geq 1$ is an integer, write
\begin{align*}
 s  \lambda & =  (1^{s m_1} 2^{s m_2}....)  &  s \cdot \lambda & = ( s^{m_1} (2s)^{m_2} ....).
\end{align*}
If $\mu = (1^{r_1} 2^{r_2}....)$ define $$\lambda \cup \mu = (  1^{m_1+r_1} 2^{m_2 + r_2}...)$$
so that $s \lambda =  \lambda \cup ... \cup \lambda$.

For $m \in \Z_{>0}$, set $\varphi_m(y) := (1-y)(1-y^2) \cdots (1-y^m) \in \Z[y]$. Then for a partition $\lambda \in \PP$, we set
\begin{align*}
a_\lambda(y) := y^{|\lambda| + 2 n(\lambda)} \prod_{d \geq 1} \varphi_{m_{d}(\lambda)}(y^{-1})
\end{align*}
which is a polynomial in $y$.

\subsection{Conjugacy classes, types and symmetric types} \label{s:symmetrictypes}
Conjugacy classes in $GL_n(\F_q)$ are classified by rational canonical forms, or equivalently \cite[IV.2]{M} by maps $\bmu : \Phi \to \PP$ of \emph{norm} $\norm{\bmu}  =n$, where
\begin{align*}
\norm{\bmu} := \sum_{f \in \Phi} d_f | \bmu(f)|.
\end{align*}
Write $c_{\bmu}$ for the conjugacy class corresponding to the map $\bmu$.  The \emph{support of $\bmu$} is defined as $$\supp \bmu := \{ f \in \Phi \, : \, \bmu(f) \neq \emptyset \}.$$ Since $\{t \pm 1\}$ play a special role in this paper, we also make use of the set difference $$ \supp'(\bmu) := \supp \bmu \setminus \{ t \pm 1 \}.$$

The \emph{type} of $\bmu$ is the map $\rho : \PP \setminus \{ \emptyset \} \to \PP$ is defined by
\begin{align*}
\rho(\lambda)& = (1^{m_1} 2^{m_2} ...), &m_d  = m_{d,\lambda} = m_{d,\lambda}(\rho)&  := | \{ f \in \Phi_d \, : \, \bmu(f) = \lambda \} | .
\end{align*}
Define the \emph{norm} of a type by 
$$  \norm{\rho} = \sum_{ \lambda \in \PP \setminus \emptyset}  \sum_{d \geq 1}   d m_{d,\lambda}  |\lambda|. $$
Note that for a given $n$, the possible types of norm $n$ are independent of $q$ for $q$ sufficiently large.  

If $\bmu$ has type $\rho$, then \cite[IV(2.7)]{M} the order of the centralizer $Z(c_{\bmu}) \leq GL_n(\F_q)$ is
\begin{align}\label{formforcontr}
|Z(c_{\bmu})| = a_{\bmu}(q) := \prod_{f \in \Phi} a_{\bmu(f)}(q^{d_f}) = \prod_{\lambda \in \PP \setminus \{ \emptyset \} } \prod_{d \geq 1} a_\lambda(q^d)^{m_{d,\lambda}(\rho)},
\end{align}
 \cite[IV(2.7)]{M}. Notice the order depends only on the type $\rho$.

We call $\bmu$ \emph{symmetric} if $\bmu(f)=\bmu(f^*)$ for all $f \in \Phi$. The \emph{symmetric type} of $\bmu$ is the tuple $\eta = (\eta_+, \eta_-, \eta_s, \eta_{p})$, where $\eta_+$, $\eta_- \in \PP$, and $\eta_s$, $\eta_{p} : \PP \setminus \{ \emptyset \} \to \PP$ are defined by
\begin{enumerate}[(i)]
\item $\bmu(t\mp1) = \eta_{\pm}$, and
\item for $\lambda \in \PP \setminus \{ \emptyset \}$, one has
\begin{align*}
\eta_s(\lambda) &:= (1^{m_{1,\lambda}^s} 2^{m_{2,\lambda}^s}...) & m_{d,\lambda}^s &:=  | \{ f \in \Phi_{2d}^s \, : \, \bmu(f) = \lambda \} |\\
\eta_{p}(\lambda) &:= (1^{m_{1,\lambda}^p} 2^{m_{2,\lambda}^p}...) &  m_{d,\lambda}^p & : = | \{ \{f,f^*\} \in \Phi_d^{p} \, : \, \bmu(f) = \lambda \} |.
\end{align*}
\end{enumerate}
We write $\bmu \in \eta$ to indicate that $\bmu$ has symmetric type $\eta$. We sometimes abuse notation and write $\eta$ for the set of all conjugacy classes $c_{\bmu}$ of symmetric type $\eta$.

\section{An explicit formula for \texorpdfstring{$F$}{Lg}}\label{s:General formula for F}\label{Explicitform}

In this section, we use Milnor's classification of orthogonal transformations over perfect fields \cite{Mi} (following Williamson \cite{W}) to derive explicit formulas for the function $F: GL_n(\F_q) \rightarrow \Z_{\geq 0}$.  What we use in later sections is the following proposition.

\begin{prop} \label{t:Fformula}
The function $F$ vanishes on the conjugacy class $c_{\bmu}$ unless $\bmu$ is symmetric (i.e  $\bmu(f) = \bmu(f^*)$ for all $f \in \Phi$). If $\bmu$ has symmetric type $\eta$, then there exists a monic polynomial $b_{\eta}(y) \in \Z[y]$ depending only on $\eta$, such that 
\begin{equation}\label{polyforF}
 F(c_{\bmu}) = b_\eta(q).
 \end{equation}
The degree of $b_{\eta}(y)$ is equal to $\tfrac{1}{2} \sum_{f = t\pm 1} \ell_{odd}(\bmu(f))  +  \sum_{f \in \supp(\bmu)}  d_f \left( n(\bmu(f)) + \tfrac{1}{2} |\bmu(f)| \right)$.
\end{prop}

 Let $V$ be a finite-dimensional vector space over a finite field $\F = \F_q$ and let $t \in GL(V)$ be a linear automorphism.  Then there is a natural decomposition
\begin{align} \label{e:primarydecomp}
V = \bigoplus_{f \in \Phi} V_f,
\end{align}
where $V_f$ is the $f$-primary component of $V$ with respect to $t$.

\begin{prop} \label{p:orthM}
Suppose $t$ is orthogonal with respect to a non-degenerate symmetric bilinear form (ndsbf) $\langle,\rangle$.  Then $V_f$ is orthogonal to all components except $V_{f^*}$.   Consequently
$$ F(t) = \prod_{f \in \Phi^s} F(t|_{V_f}) \cdot \prod_{ \{f,f^*\} \in \Phi^p} F \left( t|_{V_f \oplus V_{f^*}} \right) $$
where the first product is over self-dual irreducible factors and the second is over distinct pairs $f, f^*$.
\end{prop} 
\begin{proof}
Orthogonality is proven in \cite[Lemma 3.1]{Mi}. The consequences are immediate.
\end{proof}

\begin{prop}\label{pretoref}
Suppose that the minimal polynomial of $t \in GL(V)$ has only $f$ and $f^*$ as monic, irreducible factors where $f \neq f^*$ and $d= \deg f$. Then $F(t) = 0$ unless $t|_{V_{f}}$ is similar to $ (t^{-1})^T|_{V_{f^*}}$. If they are similar, then $F(t) =  a_{\bmu(f)}(q^{d})$. In particular, $F(t)$ equals a monic polynomial in $q$ of degree $2d \left( n(\bmu(f)) + \tfrac{1}{2} |\bmu(f)| \right)$.
\end{prop}

\begin{proof}
If $F(t) \neq 0$, then $t$ is orthogonal with respect to some ndsbf $\langle,\rangle$. By Proposition \ref{p:orthM}, $\langle,\rangle$ determines a duality pairing between  $V_{f}$ and $V_{f^*}$. If we choose a basis for $V_{f}$ and the dual basis in $V_{f^*}$ then $t$ must have the form
$$ t =   \left[ \begin{array}{cc} A & 0 \\ 0 & (A^T)^{-1} \end{array} \right]. $$
The set of ndsbfs on $V$ compatible with $t$ are those represented by a symmetric matrix of the form
$$   \left[ \begin{array}{cc} 0 & B \\ B^T & 0 \end{array} \right] $$
where $AB = BA$ and $B$ is invertible. Consequently, $F(t)$ equals the order of the centralizer of $t|_{V_{f}}$ in $GL(V_{f})$, which equals $a_{\bmu(f)}(q^{d})$ by (\ref{formforcontr}).
\end{proof}

\begin{rmk}\label{remarkonFproperty}
Propositions \ref{p:orthM} and \ref{pretoref} imply that if  $\bmu(f) \neq \bmu(f^*)$ for some $f \in \Phi$ then $F(c_{\bmu})=0$. This means $F$ is supported on conjugacy classes of symmetric type.
\end{rmk}

Suppose that $V$ is $f$-primary with $f \in \Phi^s$ (recall (\ref{definephis})). By the fundamental theorem of PIDs, there is an isomorphism of $\F[t]$-modules
$$ V = V_1 \oplus V_2 \oplus \cdots \oplus V_k  $$
for some $k$, where $V_i \cong  \frac{\F[t]}{f(t)^i \F[t]} \otimes_\F \F^{m_i}$ for some sequence of non-negative integers $m_1,...,m_k$.

\begin{prop}
Suppose $V$ is as above and $d = \deg f$. Then
$$F(t) = q^{d \sum_{ 1\leq i < j \leq k} i m_i m_j }  \prod_{i = 1}^k F(t|_{V_i}). $$
\end{prop}

\begin{proof}
Any ndsbf left invariant by $t$ restricts non-degenerately to $V_k$ \cite[Thm. 3.2]{Mi}, so it determines an orthogonal decomposition $V_k \oplus V_k^{\perp}$ and $V_k^{\perp}$ is isomorphic as $\F[t]$-module to $V_1 \oplus \cdots \oplus V_{k-1}$.  The number of complements of $V_k$ in $V$ as a $\F[t]$-module is equal to the number $\F[t]$-module splittings of $ 0 \rightarrow V_k \rightarrow V \rightarrow V/V_k \rightarrow 0$ which is equal to $q^{d \sum_{i=1}^{k-1} i m_i m_k}$. Therefore
$$ F(t) =  q^{d \sum_{i=1}^{k-1} i m_i m_k} F(t|_{V_k}) F( t|_{V_1 \oplus \cdots V_{k-1}}).$$
The formula follows by induction.
\end{proof}

\begin{prop} \label{p:Fsddeg1}
Suppose that $f \in \Phi^s_1$ and $V = V_i \cong \frac{\F[t]}{f^i \F[t]} \otimes_\F \F^m$. Then  $$F(t) =  \begin{cases} q^{(i m^2 + m)/2 } \prod_{j=1}^{m/2} (1- q^{1-2j})  & \text{if $i$ is odd and $m$ is even} \\ 
q^{(im^2 + m)/2 } \prod_{j=1}^{(m+1)/2} (1- q^{1-2j})  & \text{if $i$ is odd and $m$ is odd} \\   
 q^{i m^2/2 } \prod_{j=1}^{m/2} (1- q^{1-2j})   & \text{if $i$ is even and $m$ is even}\\
 0 & \text{if $i$ is even and $m$ is odd}  \end{cases}.   $$
\end{prop}

\begin{proof}
Let $f \in \Phi^s_1 =  \{t\pm 1\} $ and let $\Delta := t - t^{-1} = (t-1)(t+1) t^{-1}$. If $\langle,\rangle$ is a ndsbf on $V$ for which $t$ is orthogonal then $\Delta$ is skew adjoint in the sense that $\langle \Delta v, w \rangle = - \langle v, \Delta w\rangle$. This determines a non-degenerate bilinear form on $V/f(t)V \cong \F^m_q$ defined by $$(v) \cdot (w) := \langle \Delta^{i-1} v, w\rangle.$$ The form $\cdot$ is symmetric if $i$ is odd and antisymmetric if $i$ is even. Therefore $$F(t) = \alpha \beta$$ where $\alpha$ is the number of nondegenerate $(-1)^{i-1}$-symmetric form on $V/f(t)V$ and $\beta$ is the number of $t$ compatible ndsbf on $V$ associated to a given form on $V/f(t)V$. 

If $i$ is odd, then $\alpha$ equals the number of ndsbf on $\F_q^m$.  This set decomposes into a union of two $GL_m(\F_q)$ orbits, so 
$$ \alpha  =  \sum_B  |GL_m(\F_q)|/ |O(B)|  $$
summing over representatives $B$ of the two equivalence classes of ndsbf.  

If $i$ is even, then $\alpha$ is equal to the number of non-degenerate skew-symmetric bilinear forms on $\F_q^m$.  Therefore
\begin{align*}
\alpha  = \begin{cases}
|GL_{m}(\F_q)| / |Sp_{m}(\F_q)|  & \text{if $m$ is even} \\
0 & \text{if $m$ is odd} .
\end{cases}
\end{align*}

Substituting the orders of groups (see e.g. \cite{LS}) and we obtain
\begin{align} 
\alpha   = \begin{cases}  q^{(m^2+m)/2} \prod_{j=1}^{m/2} (1- q^{1-2j}) & \text{ if $i$ is odd and $m$ is even} \\ q^{(m^2 +m)/2} \prod_{j=1}^{(m+1)/2} (1- q^{1-2j}) & \text{ if $i$ is odd and $m$ is odd}\\
  q^{(m^2-m)/2} \prod_{j=1}^{m/2} (1-q^{1-2j} ) & \text{ if $i$ is even and $m$ is even}\\
0 & \text{ if $i$ is even and $m$ is odd.}
\end{cases} 
\end{align}

We turn now to $\beta$. Given an ndsbf $\langle,\rangle$, choose a basis $(v_1),...,(v_m)$ of $V/f(t)V$.  A choice 
of representatives $v_1,...,v_m \in V$  will be called a \emph{lift}. The number of lifts is $q^{(i-1)m^2}$. Any lift extends to a basis of $V$, 
\begin{equation}\label{extbass}
\{ \Delta^{s} v_k| s \in \{0,...,i-1\}, k \in \{1,...,m\}\}.
\end{equation} 
By \cite[Thm. 3.4]{Mi}, lifts exist that satisfy  
$$ \langle \Delta^s v_k, \Delta^t v_l\rangle = \begin{cases} (-1)^t (v_k) \cdot (v_l) & \text{ if $s+t = i-1$} \\ 0 & \text{ otherwise.} \end{cases} $$ 
Call such a lift a \emph{standard lift} for $\langle, \rangle$. It follows that $\beta$ equals the number of lifts divided by the number of standard lifts for a given $\langle ,\rangle$. It remains to count the number of standard lifts for a given non-degenerate form $\cdot$ on $V/f(t)V$.

Fix a particular standard lift $v_1,...,v_m$. If $\cdot$ is described by a $(-1)^{i-1}$-symmetric $m\times m$-matrix $A$, then in terms of the basis (\ref{extbass}), $\langle,\rangle$ is described by the symmetric $(im) \times (im)$-matrix
\begin{align}\label{standardliftform}
X := \left[ \begin{array}{ccccccc}
0 & 0 & ... &(-1)^{i-1} A \\
0 &0 &... & 0\\
\vdots &\vdots &\vdots &\vdots \\
0& -A& ... & 0\\
A & 0 &...  & 0 \end{array}
\right].
\end{align}
Then another lift $v_1',...,v_m'$ will be standard if and only if the change of basis matrix sending $\Delta^s v_k \mapsto \Delta^s v_k'$ is a lower triangular matrix of the form
\begin{align*}
Y: = \left[ \begin{array}{ccccccc}
I& 0 & 0 & ... &0 \\
B_1 &I & 0 &... & 0\\
B_2& B_1& I &... & 0\\
...&...&... &...& ...\\
B_{i-1} & B_{i-2} &B_{i-3}  & ... & I \end{array}
\right].
\end{align*}
and satisfies $ Y^T X Y = X$, or equivalently

\begin{eqnarray*}
 B_1^T A - A B_1 &=& 0 \\
 B_2^TA - B_1^TAB_1+ A B_2 &=& 0\\
 \vdots \\
 B_{i-1}^T A - B_{i-2}^T A B_1+...+(-1)^{i-1} A B_{i-1} &=&0\\
 \end{eqnarray*}
 Using the fact that $A$ is $(-1)^{i-1}$-symmetric, these can be rewritten
 \begin{eqnarray*}
 (-1)^{i-1} (AB_1)^T - A B_1 &=& 0 \\
 (-1)^{i-1}(AB_2)^T + A B_2 &=&  B_1^TAB_1\\
 \vdots \\
  (-1)^{i-1} (AB_{i-1})^T +(-1)^{i-1} A B_{i-1} &=& B_{i-2}^T A B_1+... \\
 \end{eqnarray*}
The first equation gives $m (m  - (-1)^{i+1}) /2$ independent linear equations for entries of $B_1$.  Given a solution to this, the second equation gives $m (m  - (-1)^{i+2}) /2$ independent linear equations for entries of $B_2$, and so on. Altogether, the number of independent linear equations for $Y$ is $m^2(i-1)/2$ if $i$ is odd and $(m^2(i-1) + m)/2$ if $i$ is even. It follows that
 
$$ \beta = \begin{cases} q^{m^2(i-1)/2} & \text{ if $i$ is odd}\\
q^{(m^2(i-1) +m)/2} & \text{ if $i$ is even}. \end{cases}$$

\end{proof}

If $f \in \Phi_{2d}^\s$, then $\E := \F[t]/ f(t) \F[t]$ is a field extension of degree $2d$ over $\F$. There is a unique automorphism $\bar{\cdot}$ of $\E$ over $\F$ which sends the class of $t$ to it's multiplicative inverse.  A \emph{Hermitian form} on an $\E$-vector space $W$ is a non-degenerate $\F$-bilinear pairing $h: W \times W \rightarrow \E$ which is $\E$-linear in the first entry and satisfies $h(u,v) = \overline{h(v,u)}$.

\begin{prop}
Suppose that $f \in \Phi_{2d}^s$, and $V = V_i \cong \frac{\F[t]}{f(t)^i \F[t]} \otimes_\F \F^m$. Then $$F(t) =  q^{idm^2}  \prod_{j =1}^m \left(1 +(-1)^j q^{-dj} \right).$$
\end{prop}

\begin{proof}
By \cite[Thm. 3.3]{Mi}, the ndsbf on $V$ which are compatible with $t$ are classified as follows. If $\langle,\rangle$ is fixed by $t$ then the operatior $s(t) := f(t)/t^d$ is adjoint in the sense that $\langle s(t) v, w\rangle = \langle v, s(t)w\rangle$. The quotient space $ V/f(t) V$ is isomorphic to $\E^m$ where $\E = \F[t]/f(t)\F[t]$. The quotient $V/f(t) V$ admits one and only one Hermitian inner product $\cdot$ such that
$$\Tr_{\E/\F} ((v) \cdot (w)) = \langle s(t)^{i-1} v, w\rangle, $$
where $\Tr_{\E/\F}$ is the field trace. 

Thus $$F(t) = \alpha \beta$$ where $\alpha$ is the number of Hermitian inner products on $\E^m$ and $\beta$ is the number of $t$-compatible ndsbf associated to each Hermitian inner product. Since $\E^m$ only admits one Hermitian form up to change of basis (\cite{Mi} Example 1), we deduce using \cite{LS} that
$$\alpha =  \frac{|GL_m(\E)|}{ | U_m(\E)|} = q^{ d m^2}  \prod_{\gamma =1}^m(1 +(-1)^\gamma q^{-d\gamma}). $$ 

To calculate $\beta$, let $\langle,\rangle $ be given and choose an $\E$-basis $(v_1),...,(v_m) \in V/f(t)V$ which is orthonormal with respect to $\cdot$. There are $q^{2dm^2(i-1)}$ possible choices of representatives $v_1,...,v_m \in V$ and each such choice determines a basis $ \{ t^a s(t)^b v_k| a \in \{0,..., 2d-1\}, b \in \{0,...,i-1\}, k\in\{1,...,m\}\} $ for $V$. The representatives $v_1,...,v_m \in V$ can be chosen  so that $\langle,\rangle $ has standard form
$$  \langle t^a s(t)^b v_k , t^{a'} s(t)^{b'} v_l \rangle  =   \begin{cases} \Tr_{\E/\F}( (t^av_k) \cdot (t^{a'} v_l) ) & \text{ if $b+b' =i-1$ } \\   
 0 & \text{ ~~~~~~if $|a-a'| < d$ and $b+b' \neq i-1$ } 
 \end{cases} $$
with the remaining pairings determined uniquely from these using an induction argument. We call such a choice of representatives $v_1,...,v_m$ \emph{a standard lift} with respect to $\langle,\rangle $. Therefore $\beta$ equals $q^{2dm^2(i-1)}$ divided by the number of standard lifts for a given $\langle,\rangle $. 

Let $v_1,...,v_m$ be a choice of standard lift for $\langle,\rangle $. Any other choice of representative $v_1'$ for $(v_1)$ must be of the form $$v_1' = v_1 +\sum_{k=1}^m \sum_{j=1}^{i-1} a_{j,k}(t) s(t)^j v_k$$ where $a_{j,k}(t)$ is a polynomial of degree at most $2d-1$.  In order for $v_1'$ to extend to a standard lift, Milnor shows that $a_{1,1}(t) +a_{1,1}(t^{-1})$ must descend to zero in $\E =\F[t]/f(t) $. Since the kernel of the map $\E \rightarrow \E$ sending $e \mapsto e+ \bar{e}$ has order $q^d$, this means there are $q^d$ possible choices for $a_{1,1}(t)$. Similarly, once $a_{1,1}(t)$ is chosen, there are $q^d$ choices for $a_{2,1}(t)$ and so on. For $k >1$ the polynomials $a_{j,k}(t)$ can be chosen arbitrarily so there are $q^{2d}$ choices each.  Altogether there are $$q^{(d +2d(m-1))(i-1)} = q^{d(i-1)(2m-1)}$$ choices of $v_1'$ that extend to a standard lift.  The ndsbf $\langle,\rangle $ restricts to a non-degenerate form on $span\{ t^a s(t)^b v_1'\}$ and the remaining representatives $v_2',...,v_m'$ must be chosen from its orthogonal complement. Using induction we deduce that there are $ q^{ d(i-1)( (2m-1)+(2m-3) + ... + 1) }=  q^{d(i-1)m^2}$ choices of standard lift. It follows then that $\beta = q^{dm^2(i-1)}$ which concludes the proof. 
\end{proof}

\begin{proof}[Proof of Proposition \ref{t:Fformula}]
Assemble the results of this section to identify the leading order term in $F(c_{\bmu})$.
\end{proof}

\section{Characters of \texorpdfstring{$GL_n(\F_q)$}{Lg}}\label{Characters of}

In this section we recall the classification of irreducible characters of $GL_n(\F_q)$ due to Green \cite{Gr1} and prove Proposition \ref{bigtechnical} concerning sums of character values over conjugacy classes of fixed symmetric type. Our presentation borrows from \cite{DM, Gr1, M}.

It is helpful to first consider the symmetric group $S_n$, which can morally be thought of as $GL_n(\F_1)$. The conjugacy classes $c_\mu$ are classified by partitions $\mu \in \PP_n$ determined by the disjoint cycle decomposition. If $\mu = (1^{r_1} 2^{r_2} ....k^{r_k})$ then the stabilizer of any representative of $c_\mu$ has order
$$ z_\mu := r_1!\dots r_k! 1^{r_1}\dots k^{r_k}.$$
The irreducible characters $\chi_\lambda$ of $S_n$ are also classified by partitions $\lambda \in \PP_n$.  We write $$\chi_\mu^\lambda := \chi_\lambda(c_\mu).$$

Now suppose that $q$ is a prime power, and recall that $ M_n  :=   \F_{q^n}^\times$. When $n|m$, the norm map $\Nm_{m,n} : M_m \to M_n$ is defined $\Nm_{m,n}(x) =x \cdot x^{q^n}... \cdot x^{q^{(m-1)n}} = x^{ (q^m-1)/(q^n-1)}$.  Denote the character groups $L_n  := \Hom(M_n, \C^\times)$ and define the direct limit  $$L := \lim_{\to} L_n = \bigcup_n L_n,$$ using the dual maps $\Nm_{m,n}^*: L_n \hookrightarrow L_m$. These are injective and we treat them as subset inclusions.
If  $x \in M_n$ and $\gamma \in L_n$, we define
\begin{align*}
\langle \gamma, x \rangle_n := \gamma(x).
\end{align*}
Note that if $n |m$ and $\gamma \in L_n \subset L_m$ and $x \in M_m$, we have $\langle \gamma, x \rangle_m =  \langle \gamma, x^{ (q^m-1) /(q^n-1)} \rangle_n$. The Frobenius automorphism of $M$ induces one on $L$ and we write
$$\Theta = \bigcup_{d\geq 1} \Theta_d := L/Frob$$ where $\Theta_d$ is the set of orbits $\theta \subseteq L$ of order $d$. Observe
\begin{align*}
\theta = \{ \gamma, \gamma^q, \ldots, \gamma^{q^{d-1}} \} = \{ \gamma^q,...,\gamma^{q^d} \}
\end{align*}
for any $\gamma \in \theta \subseteq L_d \setminus (\cup_{i|d} L_i)$. We say  $\theta$ has \emph{degree} $d_{\theta} := |\theta|$.
Given a map $\Lambda : \Theta \to \PP$ of finite support define the \emph{norm}
\begin{align*}
\norm{\Lambda} := \sum_{\theta \in \Theta} d_{\theta} | \Lambda(\theta) |.
\end{align*}
The irreducible characters of $\GL_n(\F_q)$ are in one-to-one correspondence maps $\Lambda$ of norm $n$. We use notation $\Lambda = ( \theta_1^{\lambda_1} \theta_2^{\lambda_2} ...)$  to mean $\Lambda(\theta_i) = \lambda_i$. A character $\chi_\Lambda$ is called \emph{primary} if $\Lambda = (\theta^{\lambda})$ is supported on a single $\theta \in \Theta$.  
The \emph{type} of $\Lambda$ is the map $\tau : \PP \setminus \{ \emptyset \} \to \PP$ defined by
\begin{align*}
\tau(\lambda)& = (1^{m_{1,\lambda}} 2^{m_{2,\lambda}} ...), &m_{d,\lambda} = m_{d,\lambda}(\tau) &  := | \{ \theta \in \Theta_d \, : \, \Lambda(\theta) = \lambda \} | .
\end{align*}

We will later use a more refined notion of type valid if $q$ is odd. If $\theta \in \Theta_d$ and $\gamma \in \theta$, set 
$$ \langle \theta, -1\rangle_d := \langle \gamma, -1\rangle_d \in \{\pm 1\}$$
which is well-defined independently of the choice of $\gamma \in \theta$. For each $d \geq 1$ partition $\Theta_d = \Theta_d^+ \cup \Theta_d^-$ where 
$$\Theta_d^{\pm} := \{ \theta \in \Theta_d | \langle \theta, -1\rangle_d = \pm 1\}.$$
The \emph{signed type} of $\Lambda$ is the pair $\sigma^+, \sigma^-: \PP \setminus \{ \emptyset \} \to \PP$ defined by 
\begin{align}\label{signedtype}
\sigma^{\pm}(\lambda)& = (1^{m_{1,\lambda}^{\pm}} 2^{m_{2,\lambda}^{\pm}} ...), &m_{d,\lambda}^{\pm} &  := | \{ \theta \in \Theta_d^{\pm} \, : \, \Lambda(\theta) = \lambda \} | .
\end{align}

Our main technical result in this section is the following. 

\begin{prop}\label{bigtechnical}
Given an irreducible character $\chi_{\Lambda}$ and a symmetric type $\eta$, there is an equality
\begin{equation}\label{sumup}
 h_{\Lambda,\eta}(q)  = \sum_{ \bmu \in \eta}  \chi_\Lambda (c_{\bmu})
 \end{equation}
where $h_{\Lambda,\eta}(y) \in \C [y]$ is a polynomial of degree no greater than \footnote{We define the degree bound in terms of a representative $\bmu \in \eta$ instead of $\eta$ directly since this is more convenient for later application.}
\begin{equation}\label{degreebound}
n(\bmu(t+1))+ n(\bmu(t-1)) +  \sum_{f \in \supp'(\bmu)} d_f \left( n(\bmu(f)) + \tfrac{1}{2} \right),
\end{equation}
whose coefficients are uniformly bounded by a constant that depends only on $\eta$ and the type of $\Lambda$. 
\end{prop}

\begin{rmk}\label{qnif}
The point of Proposition \ref{bigtechnical} is that for large $q$ the sum (\ref{sumup}) is dominated by the leading order term. This will permit us to calculate multiplicities by taking limits $q \rightarrow \infty$.
\end{rmk}

\begin{rmk}
The informal explanation for (\ref{degreebound}) is that each $\chi_\Lambda (c_{\bmu})$ equals a polynomial expression in $q$ of degree bounded by $\sum  d_f n(\bmu(f))$ (see Lemma \ref{p:degbound}), while the number of terms in the sum is a polynomial of degree $\tfrac{1}{2}  \sum_{f \in \supp'(\bmu)} d_f$. The full proof involves an inclusion-exclusion argument.
\end{rmk}

We begin with some preliminaries before stating Green's character formula for $\chi_\Lambda (c_{\bmu})$. Let $\theta \in \Theta$, let $f \in \Phi$, let $x \in M$ be a root of $f$, and let $e$ be a positive integer such that $d_f  | d_{\theta} e$. Define 
\begin{eqnarray*}
S^{\theta}_e(f) = S^{\theta}_e(x)  & := &  \sum_{ \gamma \in \theta}  \langle \gamma, x \rangle _{d_{\theta} e}
\end{eqnarray*}
(the sum is independent of the choice of root $x$ of $f$). If $\lambda =(\lambda_1 \geq \lambda_2 \geq ...)  \in \PP$ is a partition such that every block of $ d_{\theta} \cdot \lambda$ is divisible by $d_f$, then define
$$  S_{\lambda}^{\theta}(f) = S_{\lambda}^{\theta}(x)  :=  \prod_{i=1}^{\ell(\lambda)} S^{\theta}_{\lambda_i}(f). $$ 
For example, one can check the identities
\begin{align}\label{shiftids}
S^{\theta}_{\lambda}(t-1)& = d^{\ell(\lambda)}_{\theta} & S^{\theta}_{\lambda}(t+1) &= d^{\ell(\lambda)}_\theta \langle\theta,-1\rangle_{d_\theta}^{|\lambda|}.
\end{align}

\begin{lem}\label{S sum lemma}
Let $g = gcd( \lambda_1,...,\lambda_{\ell(\lambda)})$.  Then $S_{\lambda}^{\theta}$ determines a character
$$  M_{d_{\theta} g} \rightarrow \C^\times$$ 
of degree $d_{\theta} \ell(\lambda)$ which is constant along Frobenius orbits. 
\end{lem}

\begin{proof}
If $\gamma \in \theta$ and $x \in  M_{d_{\theta} g}$ and $g$ divides $e$, then
$$ S^{\theta}_e(x) =  \sum_{i=1}^{d_{\theta}} \langle \gamma^{q^i}, x \rangle _{d_{\theta}e} = \sum_{i=1}^{d_{\theta}} \langle \gamma, x^{ q^i (q^{d_{\theta}e} -1)/ (q^{d_{\theta}} -1)}\rangle _{d_{\theta}}$$
which is a Frobenius invariant character of degree $d_\theta$. Since $M_{d_{\theta} g}$ is an abelian group, the product
$$  S_{\lambda}^{\theta}(x) =  \prod_{i=1}^{\ell(\lambda)} S_{\lambda_i}^{\theta}(x)$$
is a Frobenius invariant character of degree $d_\theta \ell(\lambda)$.
\end{proof}

Given partitions $\lambda, \mu \in \PP_n$, the Green polynomial $ Q^{\lambda}_{\mu}(q) \in \Z[q]$  was defined by Green \cite[\S 4]{Gr1}.

\begin{lem}\label{Leading order Q}
Let $\mu$ and $\lambda$ be partitions of $n$ with $\mu=  \{1^{r_1}2^{r_2}...\}$. The Green polynomial $Q^{\lambda}_{\mu}(q)$ has degree less than or equal to $n(\lambda)$. If equal, then the leading coefficient is $(-1)^{\ell_{ev}(\mu)}c^{\lambda}_{(1),...,(1), (1^2),...,(1^2),..}$ where $(1^i)$ occurs $r_i$ times in the subscript. Here the $c^{\lambda}_{(1),...,(1), (1^2),...,(1^2),..}$ is a Littlewood-Richardson coefficient. 
\end{lem}

\begin{proof}
The degree bound is stated in \cite[Lemma 4.3]{Gr1}, but it is convenient to recall the proof. By definition
$$Q^{\lambda}_{\mu}(q) = \sum g^{\lambda}_{\rho_1, \rho_2,...}(q) k(\rho_1,q) k(\rho_2, q)...$$ summed over sequences $(\rho_i)$ where $\rho_1,..., \rho_{r_1}$ are partitions of $1$,  $\rho_{r_1+1},..., \rho_{r_1+r_2}$ are partitions of $2$ and so on. Here  $$k(\rho_i, q) = (1-q)...(1-q^{l(\rho_i)-1})$$
which implies $\deg( k(\rho_i, q)) \leq n(\rho_i)$ with equality if only if $\rho_i = (1^{|\rho_i|})$. The $g^{\lambda}_{\rho_1, \rho_2,...}(q)$ are Hall polynomials and by Hall's Theorem \cite[Thm. 4]{Gr1} 
\begin{equation}\label{gfirstord}
 g^{\lambda}_{\rho_1, \rho_2,...}(q)  = c^{\lambda}_{\rho_1, \rho_2,...}q^{n(\lambda)-n(\rho_1)-...} + \text{lower order terms}
\end{equation}
from which the result follows.   
\end{proof}

\begin{prop}[Green's Character Formula \cite{Gr1}]\label{GCF}
If $\Lambda = ( \theta^{\lambda})$ is a primary character and $c_{\bmu}$ is a conjugacy class then
\begin{equation}\label{primarycharacterformula}
 \chi_{\Lambda} (c_{\bmu}) =  (-1)^{n- |\lambda| }  \sum_{\pi \vdash |\lambda|}\chi^{\lambda}_{\pi} \sum_{\substack{ \brho: \Phi \rightarrow \PP,\\  |\brho(f)|= |\bmu(f)|,  \forall f \in \Phi,\\  Part(\brho) = d_{\theta} \cdot \pi} }  S^{\theta}(\brho) Q^{\bmu}_{\brho}.
\end{equation}

where 
\begin{eqnarray}
Part(\brho) &:=&  \cup_{f \in \Phi} d_f \cdot \brho(f)\\
  Q^{\bmu}_{\brho}& :=& \prod_{f \in \Phi} \frac{1}{z_{\brho(f)}}  Q_{\brho(f)}^{\bmu(f)}(q^{d_f}) \label{superQ}\\
  S^{\theta}( \brho)& :=&  \prod_{f \in \Phi} S^{\theta}_{ \frac{d_f}{d_{\theta}} \cdot \brho(f)}(f).
\end{eqnarray}

For a general irreducible character $\Lambda = ( \theta_1^{\lambda_1} ..., \theta_k^{\lambda_k})$ and conjugacy class $c_{\bmu}$ we have
\begin{equation}\label{parainduction}
 \chi_{\Lambda} (c_{\bmu}) := \sum_{|\bmu_i| = |\lambda_i|d(\theta_i) }  g^{\bmu}_{\bmu_1,...,\bmu_k}(q)\prod_{i=1}^k  \chi_{\theta_i^{\lambda_i}}(c_{\bmu_i})
\end{equation}
where 
\begin{equation}\label{gfact}
g^{\bmu }_{\bmu_1,...,\bmu_k}(q) =  \prod_{f \in \Phi} g^{\bmu(f)}_{\bmu_1(f),...,\bmu_k(f)}(q^{d_f})
\end{equation}
counts flags in $\F_q^n$ which are invariant under a fixed representative of $c_{\bmu}$ and whose subquotients are isomorphic to $(c_{\bmu_1},...,c_{\bmu_k})$.
\end{prop}

\begin{rmk}\label{basicformob}
Except for the factor $S^{\theta}( \brho)$ appearing in (\ref{primarycharacterformula}), Green's formula for  $\chi_{\Lambda}$ depends only on the type of $\bmu$.
\end{rmk}

\begin{lem} \label{p:degbound}
Let $c_{\bmu}$ be a conjugacy class of $GL_n(\F_q)$ and let $\chi_{\Lambda}$ be an irreducible character. Green's formula (\ref{parainduction}) expresses $\chi_{\Lambda}(c_{\bmu} )$ as a polynomial in $q$ of degree bounded above by
\begin{align} \label{e:degbound}
\sum_{f \in \supp \bmu} d_f  n(\bmu(f)).
\end{align} 
\end{lem}

\begin{proof}
From (\ref{superQ}) and Lemma \ref{Leading order Q} it follows that 
$$   \deg( Q^{\bmu}_{\brho}(q))  \leq \sum_{f \in \supp \bmu} d_f  n(\bmu(f))  $$
which by (\ref{primarycharacterformula}) takes care of the primary case. For the general case, note from (\ref{gfirstord}) and  (\ref{gfact}) that 
\begin{equation}
\deg (g^{\bmu }_{\bmu_1,...,\bmu_k}(q))  \leq \sum_{f} d_f \left(n(\bmu(f)) - n(\bmu_1(f)) - ... - n(\bmu_k(f))\right).
\end{equation} 
Applying this to (\ref{parainduction}) completes the proof.
\end{proof}

\begin{proof}[Proof of Proposition \ref{bigtechnical}]

We use an inclusion-exclusion argument to reduce to cyclic character sums (compare \cite[\S 3.3]{HRV}). Let $\bmu:  M  \rightarrow \PP$ be a Frobenius invariant map representing the symmetric type $\eta$. Introduce a finite set $I$ and an injective map  $\zeta_0: I \hookrightarrow M$  onto the support of $\bmu$. Define permutations $\rho$ and $\iota$ of $I$ such that  
\begin{align}\label{firstalign}
\zeta_0 \circ \rho & = Frob \circ \zeta_0  &  \zeta_0 \circ \iota  &= inv \circ \zeta_0
\end{align}
where $inv(x) = x^{-1}$.  Denote by $I/\rho$ the set of $\rho$-cycles. Consider the set of equivariant maps $$(I, M)_{eq} := \{ \zeta: I \rightarrow M | \zeta \circ \rho  = Frob \circ \zeta, \zeta \circ \iota  = inv \circ \zeta \}$$  and denote $(I, M)_{eq}'$ the subset of injective maps. Then there is a natural $z_{\eta}$-to-one surjective map from  $(I, M)_{eq}' $ to the set of maps of symmetric type $\eta$, where
$$ z_{\eta} = \prod_{d\geq 1} \prod_{\lambda \in \PP \setminus \emptyset} 2^{m_{d,\lambda}^p} d^{m_{d,\lambda}^s+m_{d,\lambda}^p} (m_{d,\lambda}^s!) (m_{d,\lambda}^p!) $$
where $m_{d,\lambda}^s, m_{d,\lambda}^p \in  \Z_{\geq 0}$ are as in \S \ref{s:symmetrictypes}. It follows from Proposition \ref{GCF} and Lemma \ref{p:degbound}  that
$$  \sum_{ \bmu \in \eta}  \chi_\Lambda (c_{\bmu}) = \frac{1}{z_{\eta}} \sum_{  \zeta \in (I, M)_{eq}' } \varphi(\zeta)$$  
where $\varphi(\zeta)$ is the value of $\chi_\Lambda$ on the conjugacy class determined by $\zeta$. By Remark \ref{basicformob}, $ \varphi( \zeta)$ is a sum of terms of the form $$ f(q) \prod_{ c \in (I/\rho)}  S_{\lambda(c)}^{\theta(c)}(\zeta(c))$$ with the $f(q) \in \C[q]$ are polynomials depending only on the type $\eta$, of degree no greater than $\sum_{f \in \supp \bmu} d_f  n(\bmu(f))$, and where $\theta: I/\rho \rightarrow \Theta$, and $\lambda:  I/\rho \rightarrow \PP$ satisfy  $$\sum_{ c \in (I/\rho)}  |\lambda(c)| d_{\theta(c)} = n.$$ To complete the proof it remains to show that the sums
$$ \sum_{\zeta \in  (I, M)_{eq}' }  \prod_{ c \in (I/\rho)}  S_{\lambda(c)}^{\theta(c)}(\zeta(c))  $$
equal polynomial expressions in $q$ with degree no greater than $\frac{1}{2} \sum_{f \in \supp'(\bmu)} d_f$ which have uniformly bounded coefficients. 

A partition of $I$ describes a surjective map $I \twoheadrightarrow J$ onto the set of blocks.  Let  $\Pi(I)_{eq}$ be the poset of partitions of $I$ for which $\theta$ is constant on blocks and both $\iota$ and $\rho$ permute blocks. Let $(J, M)_{eq}\subseteq (I, M)_{eq}$ be the subset of maps which are constant on blocks. Then by Moebius inversion we obtain the equality
\begin{equation}\label{firstmob}
  \sum_{  \zeta \in (I, M)_{eq}' } \prod_{c \in I/\rho}  S_{\lambda(c)}^{\theta(c)}(\zeta(c))  = \sum_{J \in \Pi(I)_{eq} } \mu_{eq}(J)  S(J)
 \end{equation}
where $\mu$ is the Mobius function for the poset $\Pi(I)_{eq}$ and
$$S(J) := \sum_{ \zeta \in (J, M)_{eq}} \prod_{c \in J/\rho }  S_{\lambda(c) }^{\theta(c)}(\zeta(c))$$
where $\lambda(c) =  \cup_{c'\in c} \lambda(c')$.
Next we prove that the $S(J)$ are polynomial functions of $q$ with uniformly bounded coefficients. Define $M_{n}^s := \{ x \in M_n  |~x^{-1} = x^{q^i} \text{ for some $i$} \}$.  This has cardinality  
\begin{equation}\label{Mscard}
 | M_{n}^s | = \begin{cases} 2  & \text{ if $n$ is odd} \\ q^{n/2}+1 & \text{ if $n$ is even} . \end{cases}
 \end{equation}
Then
\begin{equation}\label{SJsum}
 S(J) = \prod_{\substack{c \in J/\rho\\ \iota(c) =c }} \Big( \sum_{ x \in M_{|c|}^s }  S^{\theta(c)}_{\lambda(c)}(x) \Big) \times \prod_{ \substack{ \{c, \iota(c)\} \\ \iota(c) \neq c}}   \Big( \sum_{ x \in M_{|c|} }  S^{\theta(c)}_{\lambda(c)}(x) S^{\theta(\iota(c))}_{\lambda(\iota(c))}(x^{-1}) \Big) .
 \end{equation}

By Lemma \ref{S sum lemma}, $S^{\theta(c)}_{\lambda(c)}$ restricts to a character of degree $d_{\theta(c)} \ell(\lambda(c))$ on $M_{|c|}^s$, so
$$   \sum_{ x \in M_{|c|}^s }  S^{\theta(c)}_{\lambda(c)}(x)   = C  | M_{|c|}^s|  = \begin{cases} 2 C   & \text{ if $|c|$ is odd} \\ C(q^{|c|/2}+1) & \text{ if $|c|$ is even} . \end{cases}$$
where $0 \leq C \leq d_{\theta(c)} \ell(\lambda(c))$ is the multiplicity of the trivial character in $S^{\theta(c)}_{\lambda(c)}$. On the other hand

\begin{eqnarray}
 \sum_{ x \in M_{|c|} }  S^{\theta(c)}_{\lambda(c)}(x) S^{\theta(\iota(c))}_{\lambda(\iota(c))}(x^{-1}) &=&  \sum_{ x \in M_{|c|} }  S^{\theta(c)}_{\lambda(c)}(x) \overline{ S^{\theta(\iota(c))}_{\lambda(\iota(c))}}(x) \label{psum} \\
 & = & (q^{|c|}-1) \langle  S^{\theta(c)}_{\lambda(c)} , S^{\theta(\iota(c))}_{\lambda(\iota(c))} \rangle_{M_{|c|}} \nonumber \\
 & =& C' (q^{|c|}-1) \nonumber
\end{eqnarray}
where $0 \leq C' \leq d_{\theta(c)} d_{\theta(\iota(c))} \ell(\lambda(c)) \ell(\lambda(\iota(c)))$ by Lemma \ref{S sum lemma}. Taken together, this implies that $S(J)$ is polynomial whose coefficients are bounded by  $\prod_{c \in J/\rho} d_{\theta(c)} \ell(\lambda(c))$ and whose is degree at most $$\frac{1}{2} |\{ j \in J| \rho(j)\neq j \text{ or } \iota(j)\neq j\}|$$which is bounded above by 
$$  \frac{1}{2} |\{ i \in I| \rho(i)\neq i \text{ or } \iota(i)\neq i\}| = \frac{1}{2} \sum_{f \in \supp'(\bmu)} d_f .$$ 
\end{proof}

\section{Leading coefficients of $h_{\Lambda, \eta}$ when $\eta$ is of restricted symmetric type}

It will be important in the next section to calculate the leading coefficient of $h_{\Lambda,\eta}(y)$ some special cases of $\eta$. 

We say a symmetric map $\bmu: \Phi \rightarrow \PP$ has \emph{restricted symmetric type} if $\bmu(f) =1^{|\bmu(f)|}$ for $f \in \{ t \pm 1\}$  and  $\bmu(f) = 1^1$ for $f \in \supp'(\bmu) := \supp \bmu \setminus \{ t \pm 1 \}$. 

A restricted symmetric type  can be encoded as a tuple $\tau =(n_+, n_-, \tau_s, \tau_p)$, where $n_+$, $n_- \in \Z_{\geq 0}$ and $\tau_s, \tau_p  \in \PP$ and $ n_+ + n_{-} + 2|\tau_s| + 2|\tau_p| = n$ . This $\tau$ determines the symmetric type $\eta = (\eta_+, \eta_-, \eta_s, \eta_p)$ 
 \begin{align*}
\eta_{\pm} &:= (1^{n_{\pm}}) & \eta_s(\lambda) & := \begin{cases} 2 \cdot \tau_s & \lambda = (1) \\ \emptyset & \text{otherwise} \end{cases} & \eta_p(\lambda) & := \begin{cases} 2 \tau_p & \lambda = (1) \\ \emptyset & \text{otherwise} .\end{cases}
\end{align*}

By (\ref{degreebound}), when $\tau$ is of restricted symmetric type, the degree of $h_{\Lambda,\tau}$ is bounded by 
$$\deg h_{\Lambda,\tau}(y) \leq \tfrac{1}{2} \left((n_+-1)^2 + (n_--1)^2+n-2\right).$$
and we will call the coefficient of this term the \emph{leading coefficient} of $h_{\Lambda,\tau}$, even if it is zero.

Given a restricted symmetric type $\tau =(n_+, n_-, \tau_s, \tau_p)$ and $d \geq 1$, we may construct another restricted symmetric type $d \cdot \tau = (dn_+, dn_-, d\cdot \tau_s, d \tau_p)$ and say that $d$ \emph{divides} $d \cdot \tau$.

\begin{lem}\label{evenodd}
Let $\Lambda = (\theta^{\lambda} )$ be primary, where $\theta \in \Theta_d$ and $v:= |\lambda|$. Let $\tau = (n_+, n_-, \tau_s, \tau_p) $ be a restricted symmetric type of norm $n=  d v$. Then the leading coefficient of $h_{\Lambda, \tau}$ is zero, unless $d$ divides $(n_+, n_-, \tau_s,\tau_p)$. 

Suppose that $(n_+, n_-,\tau_s,\tau_p) = d \cdot (\tilde{n}_+, \tilde{n}_-, \tilde{\tau}_s,\tilde{\tau}_p)$ and let $\tilde{\rho} := 2 \cdot \tilde{\tau}_s \cup 2 \tilde{\tau}_p$. Then the leading coefficient of $h_{\Lambda, \tau}$ is equal to 
\begin{equation}\label{leadingcoeff}
 \langle\theta,-1\rangle_d^{ \tilde{n}_-}  \frac{ (-1)^{l(\tau_s)}}{z_{2\tilde{\tau}_s} z_{2\tilde{\tau}_p} }  \sum_{\pi \vdash v} sgn(\pi) \chi_{\pi}^{\lambda}   \sum_{\substack{ \tilde{\rho}_+ \cup \tilde{\rho}_- \cup \tilde{\rho} = \pi\\  \tilde{\rho}_+ \vdash \tilde{n}_+\\ \tilde{\rho}_- \vdash \tilde{n}_-} } \frac{ 1  }{z_{\tilde{\rho}_+} z_{\tilde{\rho}_-} }.
 \end{equation}
\end{lem} 

\begin{proof}

Let $c_{\bmu}$ be a conjugacy class having a restricted symmetric type $(n_+, n_-, \tau_s, \tau_p)$ and let $\rho = \left(2 \cdot \tau_s \right) \cup 2\tau_p$. Applying Green's character formula (Proposition \ref{GCF}) and identity (\ref{shiftids}) we have
\begin{eqnarray*}\label{primfurm}
 \chi_{\Lambda}(c_{\bmu}) &=&  (-1)^{n- v} \langle\theta, -1\rangle_d^{n_- /d}  \Big( \prod_{f  \in \supp'(\bmu)}  S^{\theta}_{d_f/d}(f) \Big) \\
 &&  \times    \sum_{\pi \vdash v}  \chi_{\pi}^{\lambda} \sum_{\substack{ \rho_+ \cup \rho_{-} \cup \rho = d \cdot \pi \\  \rho_+ \vdash n_+\\ \rho_- \vdash n_-}  } \frac{d^{\ell(\rho_+)+ \ell(\rho_{-})}}{z_{\rho_+} z_{\rho_-} }Q^{(1^{n_+})}_{\rho_+}(q) Q^{(1^{n_-}) }_{\rho_-}(q)    .
 \end{eqnarray*}
This expression is trivial unless $d$ divides $(n_+, n_-, 2 \cdot \tau_s, \tau_p)$, which we assume for the rest of the proof. By Lemma \ref{Leading order Q},  the Green polynomial $$Q^{\{1^n\}}_{\rho}(q) = (-1)^{\ell_{ev}(\rho)}q^{\binom{n}{2}} + \text{lower order terms} .$$
because $c^{1^n}_{\{1\},...\{1^2\},...} =1$. Therefore the leading coefficient of $\chi_\Lambda (c_\mu )$ as a polynomial in $q$ equals
\begin{equation}\label{intermediateform}
(-1)^{n- v} \langle\theta, -1\rangle_d^{n_-/d }  \Big( \prod_{  f \in \supp'(\bmu)}  S^{\theta}_{d_f/d}(f) \Big) 
   \sum_{\pi \vdash v}  \chi_{\pi}^{\lambda} \sum_{\substack{ \rho_+ \cup \rho_{-} \cup \rho = d \cdot \pi \\  \rho_+ \vdash n_+\\ \rho_- \vdash n_-}  }  \frac{d^{\ell(\rho_+ \cup\rho_{-})}}{z_{\rho_+} z_{\rho_-} }  (-1)^{\ell_{ev}(\rho_+\cup \rho_-)} .
 \end{equation}
The only part of (\ref{intermediateform}) that varies with $\bmu \in \tau$  is  $S^{\theta}_{d_f/d}(f) $. Comparing with (\ref{firstmob}) and (\ref{SJsum}), if $\tau_s = (1^{r_1} 2^{r_2}...)$ and $\tau_p = (1^{s_1} 2^{s_2}....)$, then
\begin{eqnarray}\label{charsumfo}
 \sum_{\bmu \in \tau} \prod_{ f \in \supp'(\bmu)}  S^{\theta}_{d_f/d}(f) & =&  \frac{1}{z_{2 \tau_s} z_{2 \tau_p}}  
\prod_{i \geq 1}  \left( \sum_{x \in M_{2i}^s} S^{\theta}_{2i/d}(x)  \right)^{r_i} \left( \sum_{ x \in M_i} S^{\theta}_{i/d}(x) S^{\theta}_{i/d}(x^{-1})\right)^{s_i}\\
&& + \text{lower order terms in $q$. \nonumber}
\end{eqnarray}
Following (\ref{psum}) we have 
\begin{eqnarray*}
\sum_{ x \in M_i} S^{\theta}_{i/d}(x) S^{\theta}_{i/d}(x^{-1}) &=&(q^i -1)   \langle S^{\theta}_{i/d}, S^{\theta}_{i/d} \rangle_{M_i} \\
&=& (q^i-1) d.
\end{eqnarray*}
Furthermore, $M_{2i}^s = \{x \in M | x^{q^i +1} =1\}$ by (\ref{Mscard}), so we have
\begin{eqnarray*}
 \sum_{x \in M_{2i}^s} S^{\theta}_{2i/d}(x)  &=&  \sum_{x \in M_{2i}^s} \sum_{ \gamma \in \theta}   \langle \gamma, x\rangle_{2i}\\
 &=& \sum_{ \gamma \in \theta}   \sum_{x \in M_{2i}^s}  \langle \gamma, x^{(q^{2i} -1)/(q^d-1) }\rangle_{d}\\
 &=&  \begin{cases} d  (q^i+1) & \text{ if $d | i$} \\ 0 & \text{ otherwise}\end{cases}
\end{eqnarray*}
since
$$  \frac{(q^{2i} -1)/(q^d-1)}{q^i+1}  = \frac{q^{i} -1}{q^d-1} \in \Z  \Longleftrightarrow \frac{i}{d} \in \Z.$$

 The leading coefficient of $h_{\Lambda, \tau}$ is therefore zero unless $d$ divides $(n_+, n_-, \tau_s,\tau_p)$. Letting $(n_+, n_-, \tau_s,\tau_p) = d \cdot (\tilde{n}_+, \tilde{n}_-, \tilde{\tau}_s,\tilde{\tau}_p)$, the leading coefficient equals
\begin{eqnarray*}
(-1)^{n- v} \langle\theta, -1\rangle_d^{n_-/d}   \frac{d^{\ell(\tau_s \cup \tau_p)}}{z_{2 \tau_s} z_{2 \tau_p}} 
   \sum_{\pi \vdash v}  \chi_{\pi}^{\lambda} \sum_{\substack{ \rho_+ \cup \rho_{-} \cup \rho = d \cdot \pi \\  \rho_+ \vdash n_+\\ \rho_- \vdash n_-}  } \frac{d^{\ell(\rho_+ \cup\rho_{-})}}{z_{\rho_+} z_{\rho_-} }  (-1)^{\ell_{ev}(\rho_+\cup \rho_-)} \\
   = (-1)^{n-v} \langle\theta, -1\rangle_d^{ \tilde{n}_-}  \frac{ 1}{z_{2\tilde{\tau}_s} z_{2\tilde{\tau}_p} }  \sum_{\pi \vdash v} \chi_{\pi}^{\lambda}   \sum_{\substack{ \tilde{\rho}_+ \cup \tilde{\rho}_- \cup \tilde{\rho} = \pi\\  \tilde{\rho}_+ \vdash \tilde{n}_+\\ \tilde{\rho}_- \vdash \tilde{n}_-} }  \frac{(-1)^{\ell_{ev}(\rho_+) + \ell_{ev}(\rho_-) }  }{z_{\tilde{\rho}_+} z_{\tilde{\rho}_-} }.
 \end{eqnarray*}
For the sign, note $$\ell_{ev}(\rho_+) + \ell_{ev}(\rho_-) \equiv \ell_{ev}(d \cdot \pi) + \ell(\tilde{\tau}_s)~mod~2.$$  
 If $d$ is odd, then $n-v$ is even and  $\ell_{ev}(d\cdot \pi) = \ell_{ev}(\pi)$.  If $d$ is even, then $n$ is even and $\ell_{ev}(d \cdot \pi) \equiv \ell_{ev}(\pi) +v~mod~2$. Both cases yield (\ref{leadingcoeff}), using $ sgn(\pi) = (-1)^{\ell_{ev}(\pi)}$.
\end{proof}

\begin{prop}\label{generalepsilon0}
If $\Lambda = (\theta_1^{\lambda_1}, ..., \theta_k^{\lambda_k})$ and $\tau$ is restricted symmetric type then  

\begin{equation}\label{hforgeneral}
h_{\Lambda,\tau}(y)  =_{l.o.t}  \sum_{ \tau_1 \cup ... \cup \tau_k = \tau}  g_{\tau_1,...,\tau_k}^\tau(y) \prod_{i=1}^k h_{(\theta_i^{\lambda_i}),\tau_i}(y) .
\end{equation}
where $=_{l.o.t}$ means equal modulo terms of less than leading order.
\end{prop}

\begin{proof}

Set $n_i:= \norm{\theta_i^{\lambda_i}} =d_{\theta_i} |\lambda_i|$ and $n:= \norm{\Lambda} = n_1+...+n_k$. Recall (\ref{parainduction}) that for $\Lambda = ( \theta_1^{\lambda_1} ,..., \theta_k^{\lambda_k})$ we have
\begin{equation*}
\chi_{\Lambda} (c) := \sum_{(c_1,...,c_k)}  g^{c}_{c_1,...,c_k}(q)\prod_{i=1}^k  \chi_{\theta_i^{\lambda_i}}(c_i).
\end{equation*}
where $c_i$ are conjugacy classes of rank $n_i$. If $c = c_{\bmu}$ has restricted symmetric type $\tau = (n_+, n_-, \tau_s, \tau_p)$, then it is principal and cannot be produced using non-trivial extensions. Therefore $g^{c_{\bmu}}_{c_1,...,c_k}(q) = 0$ unless $c_{\bmu} \cong c_1 \oplus ... \oplus c_k$ and we may write
\begin{equation}\label{easyrf}
\chi_{\Lambda} (c_{\bmu}) := \sum_{ \bmu_1 \cup ... \cup \bmu_k = \bmu}  g^{\bmu}_{\bmu_1,...,\bmu_k}(q)\prod_{i=1}^k  \chi_{\theta_i^{\lambda_i}}(c_{\bmu_i}).
\end{equation}
By (\ref{gfirstord}) and (\ref{gfact}), $g^{\bmu}_{\bmu_1,...,\bmu_k}(q)$ is monic of degree  
\begin{equation}\label{gdeg}
\deg(g^{\bmu}_{\bmu_1,...,\bmu_k}(q)):= n(1^{n_+})+n(1^{n_-}) - \sum_{i=1}^k \left(n(1^{n_{i,+}}) + n(1^{n_{i,-}})\right)
\end{equation}
where $n_{i,\pm}$ is the multiplicity of the eigenvalue $\pm 1$ in $c_{\bmu_i}$. 

Summing (\ref{easyrf}) over $\bmu$ gives
\begin{equation*}
h_{\Lambda, \tau}(q) = \sum_{ \bmu \in \tau} \chi_{\Lambda} (c_{\bmu})  =    \sum_{ \bmu_1 \cup ... \cup \bmu_k = \bmu \in \tau  } g^{\bmu}_{\bmu_1,...,\bmu_k}(q) \prod_{i=1}^k  \chi_{\theta_i^{\lambda_i}}(c_{\bmu_i}).
\end{equation*}
This sum includes tuples $(\bmu_1,...,\bmu_k)$ that send polynomials in a symmetric pair $\{f,f^*\}$  to different blocks $\theta_i^{\lambda_i},  \theta_j^{\lambda_j}$ where $\theta_i \neq \theta_j$. But the contribution of these tuples to the calculation of the leading term involves a factor equal to the character sum
\begin{eqnarray*}
\sum_{ x \in M_{d_f}} S^{\theta_i}_{d_f/d_{\theta_i}}(x) S^{\theta_j}_{d_f/d_{\theta_j}}(x^{-1}) &=& \sum_{x \in M_{d_f}} \sum_{\gamma \in \theta_i, \gamma' \in \theta_j} \langle \gamma, x\rangle_{d_f} \overline{\langle \gamma',x\rangle}_{d_f} \\
&=& 0
\end{eqnarray*}
(compare (\ref{charsumfo})). To leading order, it is therefore enough sum over $(\bmu_1,...,\bmu_k)$ for which each $\bmu_i$ has restricted symmetric type $\tau_i$ of norm $\norm{\tau_i}=n_i$.  Since $g^{\bmu}_{\bmu_1,...,\bmu_k}(q)$ depends only on type, we can write $g^{\tau}_{\tau_1,...,\tau_k}(q) = g^{\bmu}_{\bmu_1,...,\bmu_k}(q)$ when $\bmu \in \tau, \bmu_i \in \tau_i$. Then
\begin{eqnarray*}
 \sum_{ \bmu_1 \cup ... \cup \bmu_k = \bmu \in \tau  }g^{\bmu}_{\bmu_1,...,\bmu_k}(q) \prod_{i=1}^k  \chi_{\theta_i^{\lambda_i}}(c_{\bmu_i}) &=_{l.o.t.}& 
  \sum_{ \tau_1 \cup \cdots \cup \tau_k = \tau}   g^{\tau}_{\tau_1,...,\tau_k}(q) \sum_{\substack{ \bmu_i \in\tau_i,\\ i\in \{1,...,k\}}} \left(  \prod_{i=1}^k  \chi_{\theta_i^{\lambda_i}}(c_{\bmu_i}) \right) \\
& =_{l.o.t.}&   \sum_{ \tau_1 \cup \cdots \cup \tau_k = \tau}   g^{\tau}_{\tau_1,...,\tau_k}(q) \prod_{i=1}^k  \left( \sum_{\bmu_i \in\tau_i} \chi_{\theta_i^{\lambda_i}}(c_{\bmu_i}) \right)\\
& =& \sum_{ \tau_1 \cup \cdots \cup \tau_k = \tau}   g^{\tau}_{\tau_1,...,\tau_k}(q) \prod_{i=1}^k  h_{(\theta_i^{\lambda_i}), \tau_i}(q) \end{eqnarray*}
where we get further lower order terms accounting for when $supp'(\bmu_i) \cap supp'(\bmu_j) \neq \emptyset$ for $i \neq j$.  
\end{proof}

\section{Irreducible character decomposition of \texorpdfstring{$F$}{Lg}}\label{FouTrans}

Consider the character $F : GL_n(\F_q) \to \Z_{\geq 0} \subseteq \C$ of Section \ref{s:General formula for F}.  In this section, we compute the multiplicity in $F$ of irreducible characters of $GL_n(\F_q)$, for $char(q) \gg 1$. That is, we determine the coefficients $\sa_\Lambda$ in the decomposition
\begin{align*}
F = \sum_{\substack{\Lambda: \Theta \rightarrow \PP \\  \norm{\Lambda}=n  }} \sa_\Lambda \chi_\Lambda.
\end{align*}

The conjugacy classes of $GL_n(\F_q)$ are parametrized by functions $\bmu : \Phi \to \PP$  of degree $n$ (see \S \ref{s:symmetrictypes}).  By \eqref{orthonormalbasisdec}  we have
\begin{equation}\label{innerprod0}
\sa_\Lambda = \sum_{\norm{\bmu} =n }  \frac{F(c_{\bmu})  }{|Z(c_{\bmu})|} \chi_{\Lambda}( c_{\bmu}) .
\end{equation}
Since $F$ is supported on conjugacy classes of symmetric type, substituting (\ref{formforcontr}), (\ref{polyforF}), (\ref{sumup}), we deduce that 
\begin{equation}\label{innerprod}
\sa_\Lambda  = \sum_{\norm{\eta} =n }  \frac{b_{\eta}(q)}{a_\eta(q)}h_{\Lambda,\eta}(q)
\end{equation}
where the sum is over symmetric types $\eta$ with $\norm{\eta} = n$ and $a_{\eta}(y), b_{\eta}(y), h_{\Lambda,\eta}(y) \in \C[y]$ are certain polynomials introduced earlier. 

Recall that a symmetric map $\bmu: \Phi \rightarrow \PP$ has \emph{restricted symmetric type} if $\bmu(f) =1^{|\bmu(f)|}$ for $f \in \{ t \pm 1\}$  and  $\bmu(f) = 1^1$ for $f \in \supp'(\bmu) := \supp \bmu \setminus \{ t \pm 1 \}$. The following lemma implies that only restricted symmetric types contribute to (\ref{innerprod}) in the limit $q \rightarrow \infty$.

\begin{lem} \label{l:onlyrestricted}
If $\eta$ be a symmetric type with $\norm{\eta} = n$, then
\begin{equation}\label{sigmad}
 \frac{b_{\eta}(y)}{a_\eta(y)}h_{\Lambda,\eta}(y)
\end{equation}
is rational function in $y$ of non-positive degree and has degree zero only if $\eta$ is restricted. If $\eta$ is restricted, the constant term of \eqref{sigmad} agrees with the leading coefficient of $h_{\Lambda,\eta}(y)$.
\end{lem}

\begin{proof}
From (\ref{formforcontr}), Proposition \ref{t:Fformula}, and Proposition \ref{bigtechnical}, each of $a_\eta(y)$, $b_{\eta}(y)$, and $h_{\Lambda,\eta}(y)$ are equal to polynomials in $y$, so \eqref{sigmad} is a rational function in $y$. Both $a_\eta(y)$ and $b_{\eta}(y)$ are monic, so the leading coefficient is same as $h_{\Lambda,\eta}(y)$. 

The degrees of the polynomials satisfy (in)equalities
\begin{align}
\deg a_\eta(y) & =  2\sum_{f \in \supp(\bmu)} d_f \left( n(\bmu(f))+\tfrac{1}{2} |\bmu(f)|\right),  \label{zdeg} \\
\deg b_{\eta}(y) & =   \tfrac{1}{2}\left(\ell_{odd}(\bmu(t+1))+ \ell_{odd}(\bmu(t-1))\right)  +  \sum_{f \in \supp(\bmu)}  d_f \left( n(\bmu(f)) + \tfrac{1}{2} |\bmu(f)| \right), \label{fdeg} \\
\deg h_{\Lambda,\eta}(y) & \leq   n(\bmu(t+1)) + n(\bmu(t-1)) +  \sum_{f \in \supp'(\bmu)} d_f \left( n(\bmu(f)) + \tfrac{1}{2} \right).
 \end{align}
Therefore the degree of (\ref{sigmad}) bounded above by 
\begin{equation}\label{sumupdegrees}
\frac{1}{2} \left( \sum_{f = t \pm 1}  \left( \ell_{\odd}(\bmu(f)) - | \bmu(f)| \right) +  \sum_{f \in \supp'(\bmu)}  d_f \left( 1 - | \bmu(f)| \right) \right).
 \end{equation}
 which is a sum  of non-positive terms, hence non-positive. The upperbound (\ref{sumupdegrees}) is zero if and only if $\bmu(f) =1^{|\bmu(f)|}$ for  $f \in \{t \pm 1\}$  and  $\bmu(f) = 1^1$ for $f \in \supp'(\bmu)$. 
\end{proof}

For $\pi \in \PP$, define $C_\pi^+, C_\pi^- \in \C$ by 
\begin{align}\label{formulaCCpm}
C_\pi^+ & :=  \sum_{\substack{ \rho_+, \rho_-, \tau_s, \tau_p \in \PP\\ \rho_+ \cup \rho_- \cup 2 \cdot \tau_s \cup 2\tau_p   = \pi,\\ |\rho_-|\equiv 0~mod~2 }}  \frac{(-1)^{\ell(\tau_s)}}{z_{\rho_+} z_{\rho_-} z_{2\tau_s} z_{2\tau_p}}   &
C_\pi^- & :=  \sum_{\substack{ \rho_+, \rho_-, \tau_s, \tau_p \in \PP\\ \rho_+ \cup \rho_- \cup 2 \cdot \tau_s \cup 2 \tau_p  = \pi,\\ |\rho_-| \equiv 1~mod~2}} \frac{(-1)^{\ell(\tau_s)}}{z_{\rho_+} z_{\rho_-} z_{2\tau_s} z_{2\tau_p}}  .
\end{align}
and define
\begin{align}\label{formulaCD}
C_\pi & := C_{\pi}^+ + C_\pi^- & D_\pi & := C_{\pi}^+ - C_\pi^- .
\end{align}

\begin{prop}\label{primaryepsilon}
Let $\theta \in \Theta_d$  let $\lambda \vdash v = n/d$ be a partition and let $char(q) \gg 1$. Then for the primary character $\theta^{\lambda}$ we have
 \begin{eqnarray*}\label{short formula}
 \sa_{\theta^\lambda} &=& \begin{cases} \sa^+_{\lambda}:=\sum_{\pi \vdash v} sgn(\pi) C_\pi \chi_{\pi}^\lambda & \text{ if $\langle\theta,-1\rangle_d=1$} \\
 \sa^-_{\lambda} := \sum_{\pi \vdash v} sgn(\pi) D_\pi \chi_{\pi}^\lambda & \text{ if $\langle\theta, -1\rangle_d = -1$}. \end{cases}
 \end{eqnarray*}
\end{prop}

\begin{proof}
By (\ref{innerprod}) we have
\begin{equation*}
\sa_{\theta^\lambda} = \sum_{\eta \text{ symmetric type}} \sum_{\norm{\eta} =n }  \frac{b_{\eta}(q)}{a_\eta(q)}h_{\Lambda,\eta}(q)
\end{equation*}
By Lemma \ref{l:onlyrestricted} we know that 
\begin{equation}\label{restsymmtypesum}
\sa_{\theta^\lambda}  + O(q^{-1}) = \sum_{\eta  \text{ restricted symmetric type} }  \sum_{\norm{\eta} =n }  \frac{b_{\eta}(q)}{a_\eta(q)}h_{\Lambda,\eta}(q).
\end{equation}
By Remark \ref{qnif} we know that constant terms must match for $char(q) \gg 1$. Applying Lemma \ref{evenodd}, and taking a limit $q \rightarrow \infty$ we have 
\begin{eqnarray*}
\sa_{\theta^\lambda}  &=& \sum_{\substack{ n_+, n_- \geq 0, \tau_s, \tau_p \in \PP\\  |n_+|+|n_-|+ 2|\tau_s|+2|\tau_p| = v }} \langle\theta,-1\rangle_d^{ n_-}  \frac{ (-1)^{l(\tau_s)}}{z_{2\tau_s} z_{2\tau_p} }  \sum_{\pi \vdash v} sgn(\pi) \chi_{\pi}^{\lambda}   \sum_{\substack{ \rho_+ \cup \rho_- \cup 2 \cdot \tau_s \cup 2\tau_p = \pi \\ |\rho_+| = n_+, |\rho_-| = n_-} } \frac{ 1  }{z_{\rho_+} z_{\rho_-} } \\
&=& \sum_{\pi \vdash v}  sgn(\pi)\chi_\pi^{\lambda} ( C_\pi^+ + \langle\theta, -1\rangle_d C_\pi^-) .
 \end{eqnarray*}
\end{proof}

\begin{rmk}
It may seem strange not to absorb the $sgn(\pi)$ into the definition of $C_{\pi}$ and $D_{\pi}$. However we show in \S \ref{Determining multiplicities} that $C_\pi$ and $D_{\pi}$ are always non-negative. 
\end{rmk}

\begin{prop}\label{generalepsilon}
If $\Lambda = (\theta_1^{\lambda_1}, ..., \theta_k^{\lambda_k})$ then
$$ \sa_{\Lambda} = \prod_{i=1}^k \sa_{\theta_i^{\lambda_i}} . $$
\end{prop}

\begin{proof}
By Lemma \ref{l:onlyrestricted} we know that 
\begin{equation}\label{aLameqn}
\sa_{\Lambda}  + O(q^{-1}) = \sum_{\tau  \text{ restricted symmetric type} }   \frac{b_{\eta}(q)}{a_\eta(q)}h_{\Lambda,\eta}(q) .
\end{equation}
 
 Substituting (\ref{hforgeneral}) we get 
\begin{align*}
\sa_{\Lambda}  & =_{l.o.t.} \sum_{\substack{\norm{\tau_i} =n_i,\\ i=1,...,k}} \frac{b_{\tau_1 \cup ...\cup \tau_k }(q)  }{a_{\tau_1\cup...\cup\tau_k}(q)}  g^{\tau}_{\tau_1,...,\tau_k}(q)  \prod_{i=1}^k  h_{(\theta_i^{\lambda_i}),\tau_i}(q) \\
 &=_{l.o.t.} \prod_{i=1}^k  \left( \sum_{\norm{\tau_i} = n_i}   \frac{b_{\tau_i}(q)}{a_{\tau_i}(q)}  h_{(\theta_i^{\lambda_i}),\tau_i}(q) \right) = _{l.o.t.} \prod_i \sa_{\theta_i^{\lambda_i}} 
\end{align*}
where the second equivalence is a consequence of the two rational functions in $q$
\begin{align*}
\frac{b_{\tau_1 \cup ...\cup \tau_k }(q)  }{a_{\tau_1\cup...\cup\tau_k}(q)}  g^{\tau}_{\tau_1,...,\tau_k}(q)    && \prod_{i=1}^k   \frac{b_{\tau_i}(q)}{a_{\tau_i}(q)} 
\end{align*}
being monic of the same degree, which is readily verified from (\ref{gdeg}), (\ref{zdeg}), and (\ref{fdeg}). It follows that $\sa_{\Lambda} = \prod_i \sa_{\theta_i^{\lambda_i}} $ for $ char(q) \gg 1$.
\end{proof}

\section{Schur function formulas}\label{Determining multiplicities}

Proposition \ref{primaryepsilon} can be reformulated in terms of symmetric functions. Let $\Lambda_{\Z}$ be the ring of symmetric functions in variables $\{x_1,x_2,\dots \}$. For each $\pi \in \PP$ define the power function $p_\pi = \prod_i p_{\pi_i}$ where $p_n := x_1^n +x_2^n+... \in \Lambda_\Z$. 
These are related to the Schur functions $s_\lambda \in \Lambda_{\Z}$ by
\begin{equation}\label{schurchavar}
  p_\pi  =  \sum_{\pi \vdash n}  \chi_\pi^{\lambda} s_\lambda.
  \end{equation}
Both  $\{ p_\pi| \pi \in \PP\} $ and $\{s_{\lambda} | \lambda \in \PP\}$ form $\Z$-bases of $\Lambda_\Z$.

\begin{cor}
We have an equality of symmetric functions
\begin{align}\label{convoutsings}
 \sum_{\lambda \in \PP} \sa_{\lambda'}^+ s_{\lambda}  &= \sum_{\pi \in \PP} C_{\pi} p_\pi  &
 \sum_{\lambda \in \PP} \sa_{\lambda'}^- s_{\lambda} & = \sum_{\pi \in \PP}  D_{\pi} p_\pi 
 \end{align}
\end{cor}

\begin{proof}
Using (\ref{schurchavar}), Proposition \ref{primaryepsilon} is equivalent to the identities
\begin{align*} \sum_{\lambda \in \PP} \sa^+_{\lambda}s_{\lambda} &  = \sum_{\pi \in \PP} sgn(\pi) C_{\pi} p_\pi &
 \sum_{\lambda \in \PP} \sa^-_{\lambda} s_{\lambda}  &= \sum_{\pi \in \PP} sgn(\pi) D_{\pi} p_\pi
 \end{align*}
and (\ref{convoutsings}) follows by multiplying the degree $n$ summands on both sides by $s_{1^n}$ which corresponds to tensoring by the alternating representation of the symmetric group and satisfies $s_{\lambda} s_{1^n} = s_{\lambda'}$ and $p_\pi s_{1^n} = sgn(\pi) p_\pi$ .
\end{proof}

\begin{prop}
We have equalities of symmetric functions
$$  \sum_{\pi \in \PP} C_\pi p_{\pi} =  \left( \prod_{m \text{ odd}} e^{ 2\frac{p_m}{m} + \frac{p_m^2}{2m}} \right)  \left(\prod_{m \text{ even}} e^{ \frac{p_m}{m}+\frac{p_m^2}{2m}} \right) .$$
$$  \sum_{\pi \in \PP} D_\pi p_{\pi} =  \left( \prod_{\text{ $m$ odd}} e^{ \frac{p_m^2}{2m}} \right)  \left(\prod_{\text{ $m$ even}} e^{ \frac{p_m}{m}+\frac{p_m^2}{2m}} \right).$$
\end{prop}

\begin{proof}
Because $z_{1^{r_1}2^{r_2}...} = \prod_{m} z_{m^{r_m}}$ we deduce from (\ref{formulaCD}) that 
\begin{align*}
C_{1^{r_1}2^{r_2}...} & = \prod_{m} C_{m^{r_m}} & & \text{and} & D_{1^{r_1}2^{r_2}...} & = \prod_{m} D_{m^{r_m}}.
\end{align*}
so 
\begin{align*}
\sum_{\pi \in \PP} C_\pi p_{\pi}  & = \prod_m \Big( \sum_{i=0}^{\infty} C_{m^n} p_m^n\Big) & & \text{and} & \sum_{\pi \in \PP} D_\pi p_{\pi}  & = \prod_m \Big( \sum_{n=0}^{\infty} D_{m^n} p_m^n\Big) .
\end{align*}
We are reduced to identifying the generating functions $\sum_{n\geq 1} C_{m^n}t^n$ and $\sum_{n\geq 1} D_{m^n}t^n$.

If $m$ is odd,  we have
$$ C_{m^n} :=  C_{m^n}^+ + C_{m^n}^- =  \sum_{i+j +2k=n} \frac{1}{i! j! k! 2^k m^{i+j+k}},$$
$$ D_{m^n} :=  C_{m^n}^+ - C_{m^n}^- =  \sum_{i+j +2k=n} \frac{(-1)^j}{i! j! k! 2^k m^{i+j+k}},$$
so the generating functions satisfies
\begin{eqnarray*}
  \sum_{n \geq 0} C_{m^n} t^n & =& \left( \sum_{i} \frac{(t/m)^i}{i!} \right)\left(\sum_{j} \frac{(t/m)^j}{j!} \right) \left( \sum_{k} \frac{(t^2/2m)^k}{k!} \right) = e^{\frac{2t}{m} + \frac{t^2}{2m}}.\\
\sum_{n \geq 0} D_{m^n} t^n & =& \left(\sum_{i} \frac{(t/m)^i}{i!} \right) \left(\sum_{j} \frac{(-t/m)^j}{j!} \right) \left( \sum_{k} \frac{(t^2/2m)^k}{k!} \right) =  e^{ \frac{t^2}{2m}}.
 \end{eqnarray*}
 
If $m$ is even, then
$$ C_{m^n} = C_{m^n}^+ = D_{m^n} = \sum_{i+j+2k+2l= n}  \frac{(-1)^l}{i! j! k! l!2^k m^{i+j+k+l}}  $$
so
\begin{align*}
 \sum_{n \geq 0} C_{m^n} t^n = \sum_{n \geq 0} D_{m^n} t^n  = \left(\sum_{i} \frac{(t/m)^i}{i!} \right)^2 \left( \sum_{k} \frac{(t^2/2m)^k}{k!} \right) \left(\sum_{l} \frac{(-t/m)^l}{l!} \right) = e^{\frac{t}{m} + \frac{t^2}{2m}}. & \qedhere
\end{align*} 
\end{proof}

It remains to express these symmetric polynomials in the Schur function basis.  This makes use of Pieri's rule. Let $\pi \in \PP$ be a partition thought of as Young diagram. Pieri's rule states that
$$ s_\pi s_n =  \sum_\lambda s_\lambda  $$ 
summed over $\lambda$ obtained from $\pi$ by adding $n$ blocks with at most one block per column.  Dually
$$s_\pi s_{1^n} =  \sum_\lambda s_\lambda   $$
summed over $\lambda$ obtained from $\pi$ by adding $n$ blocks with at most one block per row.  

\begin{thm}
Let $\lambda \in \PP$ be a partition with transpose $\lambda = (1^{r_1}2^{r_2}...)$. Then  
\begin{equation}\label{rir2}
\sa^+_{\lambda}= (r_1+1)(r_2+1)....
\end{equation} 
and
$$ \sa^-_{\lambda} = \begin{cases} 1 & \text{ if $\lambda'$ has only even parts} \\ 0 & \text{ otherwise} \end{cases}  $$
\end{thm}

\begin{proof}
By Macdonald (\cite{M}  \S I.7, Example 11), we have an equality of symmetric functions
\begin{equation}\label{formsumlambda}
\sum_{\lambda \in \PP} s_{\lambda}  = \left(\prod_{\text{ $n$ odd}} e^{ \frac{p_n}{n} + \frac{p_n^2}{2n}} \right)  \left(\prod_{\text{ $n$ even}} e^{ \frac{p_n^2}{2n}} \right).
\end{equation}
The Schur function $s_n = h_n$ is the complete sum of monomials of degree $n$, so 
\begin{equation}\label{snid}
 \sum_{n=0}^{\infty} s_n = \prod_i (1-x_i)^{-1} = \prod_{n}  e^{ \frac{p_n}{n}}
 \end{equation}
where the second equality is verified by taking logarithms. It follows that
\begin{equation}\label{identitya+}
\left( \sum_{\lambda \in \PP} s_{\lambda} \right) \left( \sum_{n=0}^{\infty} s_n \right) =  \prod_{n \text{ odd}} e^{ 2\frac{p_n}{n} + \frac{p_n^2}{2n}} \cdot \prod_{n \text{ even}} e^{ \frac{p_n}{n}+\frac{p_n^2}{2n}} = \sum_{\pi \in \PP} C_\pi p_{\pi} =  \sum_{\lambda \in \PP} \sa_{\lambda'}^+ s_\lambda.
\end{equation}
from which (\ref{rir2}) is deduced using Pieri's rule.

For any Young diagram, there is a unique way to remove at most one block from each row to get another diagram with only even length rows. It therefore follows by Pieri's rule that
$$ \left( \sum_{\substack{\lambda \in \PP \\ \lambda \text{ is even}} } s_\lambda \right) \left( \sum_n s_{1^n} \right) = \sum_{\lambda \in \PP} s_\lambda.$$
where $s_{1^n} = e_n$ is the $n$-th elementary symmetric polynomial. Here
\begin{equation}\label{s1nid}
 \sum_n s_{1^n} = \prod_i (1+x_i) = \left( \prod_{\text{n odd}}  e^{\frac{p_n}{n}} \right) \left( \prod_{\text{n even}} e^{-\frac{p_n}{n}} \right)
 \end{equation}
where the second equality is verified by applying logarithms. Combined with (\ref{formsumlambda}) we deduce
$$    \sum_{\substack{\lambda \in \PP \\ \lambda \text{ is even}} } s_\lambda    =  \left( \prod_{\text{ $n$ odd}} e^{ \frac{p_n^2}{2n}} \right)  \left(\prod_{\text{ $n$ even}} e^{ \frac{p_n}{n}+\frac{p_n^2}{2n}} \right) = \sum_{\pi \in \PP} D_\pi p_{\pi} = \sum_{\lambda \in \PP} \sa_{\lambda'}^- s_{\lambda} . $$
\end{proof}

\section{The E-polynomial of $\mathcal{M}_n^{\tau}$}\label{E-polysec}

Recall  Corollary \ref{FormulaCor} that we are considering
\begin{eqnarray*}
E_n(q) &:=& |G_n|^{g-1}  \sum_{\chi \in \Irr G}  \frac{\chi(\xi)}{\chi(1)^g} (\sa_{\chi})^{r} =  \sum_{\norm{\Lambda}=n }   (\sa_{\Lambda})^{r} \left( \frac{|G_n|}{\chi_\Lambda(1)}\right)^{g-1}   \frac{\chi_{\Lambda}(\xi)}{\chi_{\Lambda}(1)} \\
\end{eqnarray*}
where $G_n := GL_n(\F_q)$ and we use shorthand $\alpha I_n = \alpha$ for $\alpha \in \F_q^{\times}$ .

From \cite[(3.1.1)]{HRV} we know that 
\begin{align}
\Delta_{\Lambda}(\xi) & :=   \frac{\chi_{\Lambda}(\xi )}{\chi_{\Lambda}( 1)} =  \langle \prod_{\gamma \in L} \gamma^{|\Lambda(\gamma)|}, \xi \rangle_1.
\end{align}
From \cite[(3.1.5)]{HRV}  we know
\begin{align}
\frac{|G_n| }{ \chi_{\Lambda}(I_n)} & =   (-q^{\frac{1}{2}})^{n^2}  \mathcal{H}_{\Lambda'}(q)\label{Gquotdegree} = (-q^{\frac{1}{2}})^{n^2} \prod_{\theta \in \Theta} \mathcal{H}_{\Lambda(\theta)'}(q^{d_{\theta}})
\end{align}
where for $\lambda \in \PP$, the $\mathcal{H}_\lambda(t)$ is the normalized hook polynomial
\begin{equation}\label{hookploy}
 \mathcal{H}_\lambda(t)  =   t^{- (n(\lambda) + \frac{|\lambda|}{2})} \prod (1 -t^h) 
 \end{equation}
the product indexed by boxes in the Young diagram of $\lambda$ and $h$ is the hook length of the box (see \cite{HRV} (2.47)). Despite the fractional exponents, (\ref{Gquotdegree}) is a polynomial in $q$. 

Both (\ref{Gquotdegree}) and $\sa_\Lambda$ depend only on the signed type $\sigma = (\sigma^+, \sigma^-)$ of $\Lambda$  (\ref{signedtype}) so we can rearrange our formula
\begin{eqnarray}\label{eqformula}
E_n(q) &=&   (-q^{\frac{1}{2}})^{n^2(g-1)}  \sum_{\text{signed types $\sigma$}}   (\sa_{\sigma})^{r} \mathcal{H}_{\sigma'}(q) ^{g-1}   \sum_{\Lambda \in \sigma}  \Delta_{\Lambda}(\xi).
\end{eqnarray}

\begin{lem}\label{incexclcalc}
Let $\sigma = (\sigma^+, \sigma^-)$ be a signed type for which $m_{d,\lambda}^{\pm}$ is the number of Frobenius orbits $\theta \in \Theta_d^{\pm}$ that are sent to the partition $\lambda$. Set $m_d^{\pm} := \sum_{\lambda} m_{d, \lambda}^{\pm}$, $m_{d,\lambda} = m_{d,\lambda}^+ + m_{d,\lambda}^-$, $m_d := m_d^+ + m_d^-$, and $m := \sum_d m_d$. Then
$$\sum_{\Lambda \in \sigma} \Delta_{\Lambda}(\xi) =  \begin{cases}  \pm (-1)^{m-1} \frac{\mu(d)}{d}  \frac{ (m -1)!}{\prod_{\lambda} m_{d,\lambda}!}   \frac{(q-1)}{2}  &  \text{ if } m =m_d^+ \text{ or } m= m_d^- \text{ for some $d \equiv 1~mod~2$}  \\ 0 & \text{otherwise} \end{cases}$$
where $\mu$ is the classical Moebius function
$$ \mu(d) = \begin{cases}  1 & \text{ if $d$ is square free and has an even number of prime factors} \\
-1 & \text{ if $d$ is square free and has an odd number of prime factors} \\
0 & \text{ if $d$ is not square free} \\
 \end{cases} $$
\end{lem}

\begin{proof}
We use an inclusion-exclusion argument that modifies the proof of \cite[\S 3.4]{HRV}. Let $\Lambda_0: L \rightarrow \PP$ represent the signed type $\sigma$. Choose a finite set  $I$ with an injective map $\zeta_0: I \hookrightarrow L$ onto the support of $\Lambda_0$. Define the permutation $\rho$ of $I$ by $ \zeta \circ \rho = Frob \circ \zeta $ and define functions  $$ l, n : I \rightarrow \Z_{\geq 0}$$ where $l(i)$ equals to the length of the $\rho$-cycle containing $i$, and $n(i) = | \Lambda_0( \zeta_0(i))|$.  Note that $$\sum_{i\in I} n(i) =n.$$  Denote by $I/\rho$ the set of $\rho$-cycles.  

For $a \geq 1$, define $$L^{\pm}_a := \{ \gamma \in L_a | \langle \gamma, -1 \rangle_a = \pm 1\}.$$ Note that $L_a^+ \leq L_a$ is an index two cyclic subgroup and $ (L_a^{\pm}/Frob ) \cap \Theta_a = \Theta_a^{\pm}$. Partition $I = I^+ \cup I^-$  where  $$I^{\pm} := \{ i \in I |  \zeta_0(i) \in L^{\pm}_{l(i)} \}.$$  
Note $\rho$ preserves both $I^+$ and $I^-$. Consider the set of maps $$(I, L)_{\rho} := \{ \zeta: I \rightarrow L |  \zeta \circ \rho = Frob \circ \zeta, \text{ and }\zeta(i) \in L^{\pm}_{l(i)} \text{ for } i \in I^{\pm}  \}$$  and $(I, L)_{\rho}'$ the subset of injective maps. There is a natural $z_\sigma$-to-1 surjective map from $(I, L)_{\rho}'$ onto the set $\Lambda$ of signed type $\sigma$, sending $\zeta_0$ to $\Lambda_0$, where 
\begin{equation}\label{sumup1}
z_{\sigma} := \prod_d ( d^{m_d} \prod_{\lambda}  m_{d,\lambda}^+! m_{d,\lambda} ^-!).
\end{equation}

Consider the function $\varphi: (I, L)_{\rho} \rightarrow \C^\times$ by 
$$\varphi(\zeta) = \langle \prod_{i} \zeta(i)^{n(i)}, \xi \rangle_1 $$
If we define $\Nm_{a,1}: L_a \rightarrow L_1$ by $\Nm_{a,1}(\gamma) = \gamma^{1+q+...+q^{a-1}}$, then 
$$\varphi(\zeta) = \prod_{\text{$c \in I/\rho$}} \left\langle \Nm_{l(c), 1} (\zeta(c)), \xi \right\rangle^{n(c)}_1$$
where $l(c) =l(i)$, $n(c) = n(i)$, and $\zeta(c) = \zeta(i)$ for some $i \in c$ (the value of $\varphi$ is independent of this choice). If
$$ S(I)' := \sum_{ \zeta \in (I, L)_{\rho}' } \varphi( \zeta)$$ 
then by construction we have an equality 
\begin{equation}\label{sumup2}
\sum_{\Lambda \in \sigma} \Delta_{\Lambda}(\xi) = \frac{1}{z_{\sigma}} S(I)'.
\end{equation}

A partition of $I$ describes a surjective map  $ I\rightarrow J$ onto the set of blocks.  Consider $\Pi(I)$ the poset of partitions of $I$ whose blocks are permuted by $\rho$. Denote $(J, L)_{\rho} \subseteq (I, L)_{\rho} $ the subset of maps which are constant on the blocks of $J$ and set $n(j) = \sum_{i \in j}  n(i)$ for $j \in J$. 
By the Moebius inversion formula we have
$$ S(I)' = \sum_{ J \in \Pi(I) }  \mu_{\rho}(J) S(J) $$
where $\mu_{\rho}$ is the Moebius function for the poset $\Pi(I)$ and
$$S(J) := \sum_{\zeta \in (J, L)_{\rho}} \varphi(\zeta). $$ 
For $ \zeta \in (J, L)_{\rho}$ we have $$\varphi(\zeta) = \prod_{c \in J/\rho} \langle \Nm_{l(c),1} (\zeta(c)), \xi^{n(c)} \rangle_1,$$
where for $c \in J/\rho$,  $l(c)$ is the length of the cycle in $J$. We can therefore commute the product through the sum to get
\begin{equation}\label{sumchar}
S(J) =  \prod_{c \in J/\rho}  \left( \sum_{ \gamma \in L(c)}   \langle \Nm_{l(c),1} (\gamma), \xi^{n(c)} \rangle_1 \right)
\end{equation} 
where  
$$L(c) :=  (\bigcap_{\substack{  \tilde{c} \in I^+/\rho \\ \tilde{c} \rightarrow  c}} L_{l(\tilde{c})}^+  ) \cap (\bigcap_{ \substack{  \tilde{c} \in I^-/\rho\\ \tilde{c}\rightarrow c}} L_{l(\tilde{c})}^-) \cap L_{l(c)}. $$ 
If $a$ divides $a'$, then
\begin{align}
L_{a'}^{\pm}  \cap L_{a}  &=   L_{a}^{\pm}& \text{ if $a'/a$ is odd} \label{twolineone} \\ 
L_{a'}^+ \cap L_a &= L_a \text{ and } L_{a'}^- \cap L_a = \emptyset & \text{ if $a'/a$ is even}\label{twolinetwo}
\end{align}
so each $L(c)$ must equal one of $L_{l(c)}, L^{\pm}_{l(c)},$ or $\emptyset$. Since $\xi$ is a primitive $2n$-root of unity, the character $\langle \Nm_{a,1} (\gamma), \xi^{n(c)} \rangle_1$ restricts to the trivial character on $L_a$ if and only if $2n | n(c)$. But this never happens because $n(c) \leq n$. It follows that for all $a$
$$\sum_{\gamma \in L_a}  \langle \Nm_{a,1} (\gamma), \xi^{n(c)} \rangle_1  =  0$$
so
$$\sum_{\gamma \in L_a^+}  \langle \Nm_{a,1} (\gamma), \xi^{n(c)} \rangle_1  = - \sum_{\gamma \in L_a^-}  \langle \Nm_{a,1} (\gamma), \xi^{n(c)} \rangle_1 .$$
The restriction of $\langle \Nm_{a,1} (\gamma), \xi^{n(c)} \rangle_1$ to $L_a^+$ is the trivial character if and only if $n | n(c)$. But $ n(c) \leq n$ is an equality if and only if $J$ is a singleton, so $S(J)=0$ unless $J =\{I\}$. We deduce
\begin{equation}\label{sumup3}
S(I)' =  \mu_{\rho}(\{I \}) S(\{I\}).
\end{equation}
The poset in question is identical with the one considered in \cite{HRV} so we recover the formula (from Hanlon \cite{Han})
$$ \mu_{\rho} (\{I\}) = \begin{cases}   \mu(d)(-d)^{m_d-1} (m_d-1)! & \text{ if $\rho$ has cycle type $(d^{m_d})$} \\ 0 & \text{ otherwise.} \end{cases}. $$
Comparing with (\ref{twolineone})  and (\ref{twolinetwo}) we see that if $\rho = (d^{m_d})$ then
$$S(\{I\}) = \begin{cases}  \pm (q-1)/2  & \text{ if $d$ is odd and $m=m_d = m_d^{\pm}$} \\ 0 & \text{ otherwise}  \end{cases}$$
which combined with (\ref{sumup1}), (\ref{sumup2}), (\ref{sumup3}) completes the proof. 
\end{proof}

\begin{thm}\label{vnformula}
The E-polynomial $E(\mathcal{M}_n^{\tau}; x,y)$ equals $E_n(xy)$ where
$$ E_n(q) :=  \frac{1}{2}(q-1)(-q^{\frac{1}{2}})^{n^2(g-1)} V_n(q)$$
and
\begin{equation}\label{butexp}
V_n(q) :=  \sum_{\substack{ \text{odd }d|n \\  \sum m_{\lambda} |\lambda| = \frac{n}{d}}} (-1)^{m-1}\frac{\mu(d)}{d} \frac{(m-1)!}{\prod_{\lambda} m_{\lambda}!}  \left( (\sa^+_{\sigma})^{r}   - (\sa^-_{\sigma})^{r} \right)\mathcal{H}_{\sigma'}^{g-1}(q^d)
\end{equation}
where the sum taken over odd divisors $d$ of $n$ and non-negative integers $m_{\lambda}$ for $\lambda \in \PP$ such that $\sum_{\lambda} m_{\lambda}|\lambda| = n/d$ and
\begin{align*}
m &:= \sum_{\lambda} m_\lambda & \sa^{\pm}_{\sigma}  &:= \prod_{\lambda} (\sa^{\pm}_{\lambda})^{m_\lambda}   & \mathcal{H}_{\sigma'}(q) &:= \prod_{\lambda} \mathcal{H}_{\lambda'}(q)^{m_\lambda}.
\end{align*} 
\end{thm}

\begin{proof}
The formula for $E_n(q)$ is immediate from substituting  Lemma \ref{incexclcalc} into (\ref{eqformula}). The formula for $\sa^{\pm}_{\sigma} $ is found in  Propositions \ref{primaryepsilon} and \ref{generalepsilon}.  

Since $E_n(q)$ is a polynomial expression in $q$, this establishes the hypothesis of Corollary \ref{FormulaCor0} and determines the E-polynomial.
\end{proof}

\begin{eg}
$E_1(q)$ is a single term, indexed by $ (d, \lambda^{m_\lambda}) = (1, (1)^1)$
$$E_1(q) = \frac{(q-1)}{2} |G_1|^{g-1} (2^r) = 2^{r-1} (q-1)^g  $$
in agreement with Example \ref{n=1eg}.
\end{eg}

\begin{eg}
$E_2(q)$ is a sum of three terms, indexed by  $$(d, \lambda^{m_\lambda}) \in \{ (1, (2)^1), (1,(1^2)^1), (1,(1)^2)\}.$$ Explicitly
\begin{eqnarray*}
 E_2(q) &=& \frac{(q-1)}{2}|G_2|^{g-1} \left( 2^{r}  + \frac{3^r -1}{q^{g-1}} -\frac{2^{2r-1}}{(q+1)^{g-1}} \right)\\
 &=& \frac{1}{2}(q-1)^{g} \left( 2^{r}(q^3-q)^{g-1}  + (3^r -1)(q^2-1)^{g-1} -2^{2r-1} (q^2-q)^{g-1} \right).
 \end{eqnarray*}
\end{eg}

\begin{eg}
$E_3(q)$ is a sum of seven terms, indexed by
\begin{align*}
(d, \lambda_i^{m_{\lambda_i}})  \in   \{ (3 , (1)^1), (1, (1^3)^1),(1,(3)^1), (1,(21)^1),  (1, (1)^1(2)^1),  (1, (1)^1(1^2)^1), (1, (1)^1(1)^1(1)^1)\}.
\end{align*}
Explicitly
\begin{eqnarray*}
 E_3(q) &=& \frac{(q-1)}{2}|G_3|^{g-1} \Big( 2^{r}  + \frac{4^r}{q^{3g-3}} + \frac{4^r}{(q^2+q)^{g-1}} -\frac{4^r}{(q^2+q+1)^{g-1}}   -\frac{6^r}{(q^3+q^2+q)^{g-1}} \\&&+ \frac{8^r}{3(q+1)^{g-1}(q^2+q+1)^{g-1}} - \frac{2^r}{3(q-1)^{g-1}(q^2-1)^{g-1}}   \Big) .
\end{eqnarray*}
\end{eg}

\subsection{The generating function}

The expression (\ref{butexp}) has a beautiful interpretation using plethystic algebra. Let $K = \Q(x)$ be the ring of rational functions and consider $K[[T]]$ the ring of formal power series in $T$.  The \emph{plethystic exponential} $ \mathrm{Exp}:  T K[[T]] \rightarrow 1+TK[[T]]$ is defined by the rule $ \mathrm{Exp}(V +W) = \mathrm{Exp}(V)\mathrm{Exp}(W)$ and $$\mathrm{Exp}( a x^mT^n) = (1-x^mT^n)^{-a}$$ for $a \in \Q$.  The inverse function,  $ \mathrm{Log}:  1+T K[[T]] \rightarrow TK[[T]]$ is called the \emph{plethystic logarithm}.  For each $\lambda \in \PP$ choose  $A_\lambda \in K$, setting $A_{\emptyset} = 1$. By \cite{HLRV} \S 2.3, 
$$ \mathrm{Log}\left(  \sum_{\lambda \in \PP} A_{\lambda} T^{|\lambda|} \right) := \sum_{n \geq 1} U_n T^n $$
where 
\begin{equation}\label{Logexp}
 U_n(x)  =  \sum (-1)^{m_d-1}\frac{\mu(d)}{d} (m_d-1)! \prod_{\lambda} \frac{A_{\lambda}(x^d)^{m_d}}{m_{d,\lambda}!} 
 \end{equation}
and the sum is indexed by the set of functions  
\begin{align*}
m: \Z_{>0} \times \left( \PP \setminus \{\emptyset\} \right) \rightarrow \Z_{\geq 0},  && (d,\lambda) \mapsto m_{d,\lambda}
\end{align*} satisfying $ \sum_{\lambda \in \PP}  d |\lambda| m_{d,\lambda} = n$ and we set $m_d := \sum_{\lambda} m_{d,\lambda}$.

\begin{proof}[Proof of Theorem \ref{BigThm}]

From Theorem \ref{vnformula} we are reduced to proving
\begin{equation}\label{Vnident}
 \sum_{n\geq 1} V_n (q)T^n = \mathrm{Log}  \prod_{i=0}^{\infty} \left( \frac{ \sum_{\lambda} (\sa^+_{\lambda})^r \mathcal{H}_{\lambda'}^{g-1}(q^{2^i}) T^{2^i|\lambda|}}{ \sum_{\lambda} (\sa^-_{\lambda})^r \mathcal{H}_{\lambda'}^{g-1}(q^{2^i}) T^{2^i|\lambda|}}\right)^{\frac{1}{2^i}}.
 \end{equation}
Rewrite the right hand side of (\ref{Vnident}) as
$$\sum_{i=0}^{\infty}  \frac{1}{2^i} \mathrm{Log}\left( \sum_{\lambda} (\sa^+_{\lambda})^r \mathcal{H}_{\lambda'}^{g-1}(q^{2^i}) T^{2^i |\lambda|} \right)  -   \sum_{i=0}^{\infty}  \frac{1}{2^i} \mathrm{Log}  \left( \sum_{\lambda} (\sa^-_{\lambda})^r \mathcal{H}_{\lambda'}^{g-1}(q^{2^i}) T^{2^i |\lambda|}\right) $$
Define $V_{n,+}^i(q)$ and $V_{n,-}^i(q)$ by the formula
$$\sum_{n} V_{n,\pm}^i(q) T^{n}:=   \mathrm{Log}\left( \sum_{\lambda} (\sa_{\lambda}^\pm)^r \mathcal{H}_{\lambda'}^{g-1}(q^{2^i}) T^{2^i|\lambda|} \right).$$
Our task is to prove that for $n \geq 1$,
\begin{equation}\label{vodd}
V_n(q) =  \sum_{i=0}^{\infty}  \frac{1}{2^i}\left( V_{n,+}^i(q)  - V_{n,-}^i(q) \right)
\end{equation}
where the left hand side satisfies (\ref{butexp}). Applying (\ref{Logexp}) we see that
\begin{eqnarray*}
V_{n,\pm}^0(q) &=& \sum_{d|n} \sum_{\sum m_{\lambda} |\lambda| = \frac{n}{d}} (-1)^{m-1}\frac{\mu(d)}{d} (m-1)!  \frac{ \prod_{\lambda} ((\sa_{\lambda}^\pm)^r  \mathcal{H}_{\lambda'}(q^d)^{g-1} )^{m_\lambda}}{\prod_{\lambda} m_{\lambda}!}.
\end{eqnarray*}
More generally $V_{n,\pm}^i(q) = 0$ unless $2^i|n$, in which case
\begin{eqnarray*}
V_{n,\pm}^i(q) &=& \sum_{d| \frac{n}{2^i}} \sum_{\sum m_{\lambda} |\lambda| = \frac{n}{2^id}} (-1)^{m-1}\frac{\mu(d)}{d} (m-1)!  \frac{ \prod_{\lambda} ((\sa_{\lambda}^\pm)^r  \mathcal{H}_{\lambda'}(q^{2^id})^{g-1} )^{m_\lambda}}{\prod_{\lambda} m_{\lambda}!} \\
&=& \sum_{ \substack{ 2^i | d \\ d | n } } \sum_{\sum m_{\lambda} |\lambda| = \frac{n}{d}} (-1)^{m-1}\frac{2^i \mu(d/2^i)}{d} (m-1)!  \frac{ \prod_{\lambda} ((\sa_{\lambda}^\pm)^r  \mathcal{H}_{\lambda'}(q^d)^{g-1} )^{m_\lambda}}{\prod_{\lambda} m_{\lambda}!} \\
\end{eqnarray*}
where in the second expression we reindex replacing $2^id$ with $d$. We deduce (\ref{vodd}) from the identity 
$$ \mu(d) + \mu(d/2) + ... + \mu(d/2^k) = 0$$
that holds when $d$ is even and $d/2^k $ is odd.
\end{proof}

\subsection{Combinatorial verification for special cases $g=0,1$}
If $g=0$ and $r=1$, then $\mathcal{M}_1^\tau$ is a point and $\mathcal{M}_n^\tau  = \emptyset$ so (\ref{GenFunctBig}) reduces to the identity 
\begin{equation*}\label{GenFunctBig0}
\frac{2 q^{\frac{1}{2}}}{1-q} T=  \mathrm{Log}  \prod_{k=0}^{\infty} \left( \frac{ \sum_{\lambda} \sa^+_{\lambda} \mathcal{H}_{\lambda'}^{-1}(q^{2^k}) T^{2^k|\lambda|}}{ \sum_{\lambda} \sa^-_{\lambda} \mathcal{H}_{\lambda'}^{-1}(q^{2^k}) T^{2^k|\lambda|}}\right)^{ \frac{1}{2^k}}.
\end{equation*}
Applying the plethystic exponential, this is equivalent to
\begin{equation}\label{g0telescope}
\prod_{j \geq 1} (1-q^{\frac{2j -1}{2}}T)^{-2} =   \prod_{k=0}^{\infty} \left( \frac{ \sum_{\lambda} \sa^+_{\lambda} \mathcal{H}_{\lambda'}^{-1}(q^{2^k}) T^{2^k|\lambda|}}{ \sum_{\lambda} \sa^-_{\lambda} \mathcal{H}_{\lambda'}^{-1}(q^{2^k}) T^{2^k|\lambda|}}\right)^{ \frac{1}{2^k}}.
\end{equation}
According to (\cite{M} I.3 ex.2)  we have $ \mathcal{H}_\lambda(q)^{-1} = q^{|\lambda|/2} \bar{s}_\lambda(q)$  where $s_{\lambda}(x_1,x_2,....)$ is the Schur function and $\bar{s}_\lambda(q) := s_{\lambda}(1, q ,q^2,...)$.  Therefore using (\ref{identitya+})
\begin{eqnarray*}
 \sum_{\lambda} \sa^+_{\lambda} \mathcal{H}_{\lambda'}^{-1}(q) T^{|\lambda|}  &=& \sum_{\lambda}  \sa^+_{\lambda}  \bar{s}_{\lambda'}(q) (q^{\frac{1}{2}}T)^{|\lambda|} \\
 &=& \left( \sum_{\lambda}  \bar{s}_{\lambda}(q) (q^{\frac{1}{2}}T)^{|\lambda|}\right) \left(  \sum_{n=0}^{\infty} \bar{s}_n(q)  (q^{\frac{1}{2}}T)^{n}\ \right)
  \end{eqnarray*}
 Similar considerations apply to the  $\sa^-_{\lambda}$ series and we deduce using (\ref{snid}) and (\ref{s1nid})
 \begin{eqnarray*}
 \frac{ \sum_{\lambda} \sa^+_{\lambda} \mathcal{H}_{\lambda'}^{-1}(q) T^{|\lambda|} }{\sum_{\lambda} \sa^-_{\lambda} \mathcal{H}_{\lambda'}^{-1}(q) T^{|\lambda|} } & =&  \left(  \sum_{n=0}^{\infty} \bar{s}_n(q)  (q^{\frac{1}{2}}T)^{n}\ \right)  \left(  \sum_{n=0}^{\infty} \bar{s}_{1^n}(q)  (q^{\frac{1}{2}}T)^{n}\ \right) \\
 &=& \left(   \prod_{i=1}^{\infty} (1-q^{i-\frac{1}{2}} T)^{-1}  \right) \left(  \prod_{i=1}^{\infty} (1+q^{i-\frac{1}{2}} T)   \right) \\
 &=&  \prod_{i=1}^{\infty}  \frac{(1-q^{2i-1} T)}{ (1-q^{\frac{2i-1}{2}} T)^{2} }  .
 \end{eqnarray*}
Consequently the right hand side of (\ref{g0telescope}) is a telescoping product verifying the identity. 

When $g=1$,  we have an isomorphism
$$ \mathcal{M}_n  \cong  \C^{\times} \times \C^{\times} $$
for all $n \geq 1$  (see  \cite{HRV} Theorem 2.2.17). When $r=1$, this isomorphism can be chosen so that $ \tau(\alpha, \beta) = (\beta^{-1}, \alpha^{-1})$ so that $$\mathcal{M}_n^{\tau} \cong \C^{\times},$$ and $E(\mathcal{M}_n^{\tau};q) = q-1$.  When $r=2$, this isomorphism can be chosen so that $ \tau(\alpha, \beta) = (\alpha, \beta^{-1})$ so that $$\mathcal{M}_n^{\tau} \cong \C^{\times} \times \{ \pm 1\},$$ and $E(\mathcal{M}_n^{\tau};q) = 2(q-1)$.  
Substituting into (\ref{GenFunctBig}) and exponentiating, we get identities
\begin{equation}\label{GenFunctBig2}
 \prod_{n=1}^{\infty} (1-T^n)^{-2}=  \prod_{k=0}^{\infty} \left( \frac{ \sum_{\lambda} \sa^+_{\lambda} T^{2^k|\lambda|}}{ \sum_{\lambda} \sa^-_{\lambda} T^{2^k|\lambda|}}\right)^{ \frac{1}{2^k}},
\end{equation}
\begin{equation}\label{GenFunctBig3}
\prod_{n=1}^{\infty} (1-T^n)^{-4} =  \prod_{k=0}^{\infty} \left( \frac{ \sum_{\lambda} (\sa^+_{\lambda})^2 T^{2^k|\lambda|}}{ \sum_{\lambda} (\sa^-_{\lambda})^2 T^{2^k|\lambda|}}\right)^{ \frac{1}{2^k}}.
\end{equation}
We have
\begin{equation*}
  \sum_{\lambda} \sa^-_{\lambda} T^{|\lambda|} =  \sum_{\lambda} (\sa^-_{\lambda})^2 T^{|\lambda|}   =    \prod_{n=1}  \left(  \sum_{k=0}^{\infty}   T^{2 k n} \right) 
  = \prod_{n \geq 1} ( 1 - T^{2n})^{-1}
  \end{equation*}
and 
\begin{equation*} 
 \sum_{\lambda} \sa^+_{\lambda} T^{|\lambda|}  =   \prod_{n \geq 1}  \left(  \sum_{k=0}^{\infty}  (k+1)  T^{ k n}   \right) =   \prod_{n \geq 1} ( 1-T^{n})^{-2}
\end{equation*}
so that
\begin{eqnarray*}
 \frac{ \sum_{\lambda} \sa^+_{\lambda} T^{|\lambda|}} {\sum_{\lambda} \sa^-_{\lambda} T^{|\lambda|}}   &= &  \prod_{ n \geq 1 }  (1-T^n)^{-2}(1-T^{2n})\\
 \end{eqnarray*}
and the right hand side of  (\ref{GenFunctBig2})  is a telescoping product verifying the identity. Similarly,  
\begin{equation*} 
 \sum_{\lambda} (\sa^+_{\lambda})^2 T^{|\lambda|}  =   \prod_{n \geq 1}  \left(  \sum_{k=0}^{\infty}  (k+1)^2  T^{ k n}   \right) =   \prod_{n \geq 1}  ( 1-T^{n})^{-4}(1 - T^{2n}) 
\end{equation*}
so
$$ \frac{ \sum_{\lambda} (\sa^+_{\lambda})^2 T^{|\lambda|}} {\sum_{\lambda} (\sa^-_{\lambda})^2 T^{|\lambda|}}  =  \prod_{n \geq 1}  ( 1-T^{n})^{-4} (1 - T^{2n})^2  $$
and the right hand side of  (\ref{GenFunctBig3})  is a telescoping product verifying the identity.

\section{The E-polynomial of $\mathcal{M}_{n,w}^\tau$}\label{ccm}
Recall from Corollary \ref{FormulaCor} that
\begin{eqnarray*}
E_n^k(q) &:=& |G_n|^{g-1}  \sum_{\chi_\Lambda \in \Irr G}  \frac{\chi_{\Lambda}(\xi )}{\chi_{\Lambda}(1)^{g-1}} (\ssb^+_{\Lambda})^{r-k}(\ssb^-_{\Lambda})^{k}.
\end{eqnarray*}

\begin{lem}
Suppose $\Lambda = (\theta_1^{\lambda_1},..., \theta_m^{\lambda_m})$ where  $d_{\theta_1}= ... = d_{\theta_m} =d \equiv 1~mod~2$, and $ \langle \theta_1, -1 \rangle_d =... =  \langle \theta_m, -1 \rangle_d = \pm 1$ are both constant for all $i =1,...,m$ . Then
\begin{align*}
\ssb_{\Lambda}^+ &=  \frac{  1 }{2}\left( \prod_{i=1}^m \sa^+_{\lambda_i}  +  \prod_{i=1}^m \sa^-_{\lambda_i}  \right) &
\ssb_{\Lambda}^- &=  \pm \frac{  1 }{2}\left(  \prod_{i=1}^m \sa^+_{\lambda_i} -    \prod_{i=1}^m \sa^-_{\lambda_i}  \right)
\end{align*}
\end{lem}

\begin{proof}
 Choose $\alpha \in L_1$ such that $\langle \alpha, -1\rangle_1 = -1$. Tensoring by $\chi = \chi_{\alpha^n}$ sends $(\theta_1^{\lambda_1},..., \theta_m^{\lambda_m})$ to $ ( (\alpha \theta_1)^{\lambda_1},..., (\alpha \theta_m)^{\lambda_m})$.  Since $d_{\theta_i} = d$ is odd, we have (see \cite{KG})
$$ \langle  \alpha \theta_i, -1 \rangle_d  = \langle \alpha, -1 \rangle_d  \langle \theta_i, -1 \rangle_d  =    \langle \alpha, -1\rangle_1^d  \langle \theta_i, -1 \rangle_d  = - \langle \theta_i, -1 \rangle_d.$$
The formula follows from (\ref{Ftilde}) and our earlier calculation of $\sa_\Lambda$. 
\end{proof}

The following is proven in similar fashion to Theorem \ref{vnformula}. 

\begin{thm}\label{conncompformthm}
Suppose that $(\Sigma, \tau)$ is genus $g$ Riemann surface with anti-holomorphic involution $\tau$ such that $ \Sigma^\tau = \coprod_r S^1$ is non-empty. Let $w: \pi_0(\Sigma^\tau) \rightarrow \{\pm 1\}$ send $k$ many elements to $-1$ where $k$ is odd. Then the E-polynomial of the path component $\mathcal{M}_{n,w}^\tau$ equals $E(\mathcal{M}_{n,w}^\tau; x,y) = E_n^k(xy)$ where

$$ E_n^k(q) :=   \frac{1}{2^r}  (q-1)(-q^{\frac{1}{2}})^{n^2(g-1)}  V_{n}^k(q) $$
and
\begin{equation}\label{Vnkform}
V_{n}^k(q) := \sum_{ \substack{ \text{ odd }d | n \\ \sum m_{\lambda} |\lambda| = \frac{n}{d} } }  (-1)^{m-1}\frac{\mu(d)}{d} \frac{(m-1)!}{\prod_{\lambda} m_{\lambda}!} \left( \sa^+_\sigma   + \sa^-_\sigma \right)^{r-k} \left(  \sa^+_{\sigma}  - \sa^-_{\sigma}\right)^k  \mathcal{H}_{\sigma'}^{g-1}(q^d)
\end{equation}
\begin{align*}
m &:= \sum_{\lambda} m_\lambda & \sa^{\pm}_{\sigma}  &:= \prod_{\lambda} (\sa^{\pm}_{\lambda})^{m_\lambda}   & \mathcal{H}_{\sigma'}(q) &:= \prod_{\lambda} \mathcal{H}_{\lambda'}(q)^{m_\lambda}.
\end{align*} 
\end{thm}

Using the identity
$$ (x + y)^{r-k} (x - y)^k = \sum_{j=0}^r c_j x^{r-j} y^j$$
$$ c_j := \sum_{l=0}^j (-1)^l { k \choose l} { r-k \choose j-l}   $$
we derive the generating function.
\begin{cor}
With $V_n^k(q)$ as above we have
$$ \sum_{n=1}^{\infty} V_n^k(q)T^n =  \mathrm{Log} \prod_{i=0}^\infty \prod_{j=0}^r \left( \sum_{\lambda}  (\sa_\lambda^+)^{r-j} (\sa_\lambda^-)^j \mathcal{H}_{\lambda'}^{g-1}(q^{2^i})  T^{2^i|\lambda|}\right)^{c_j/2} .$$ 
\end{cor} 

Let $\mathcal{A}_0^\tau \cong (\C^{\times})^g$ be the identity component of $\mathcal{A}^\tau$ (see \S \ref{Aref}).

\begin{cor}\label{euler char}
If $g \geq 2$, then the Euler characteristic of $ \mathcal{M}_{n,w}/\!/\mathcal{A}_0^\tau $ is zero if $n$ is even and $\mu(n) n^{g-2}$ if $n$ is odd. 
\end{cor}

\begin{proof}
The Euler characteristic equals 
$$ \chi ( \mathcal{M}_{n,w}/\!/\mathcal{A}_0^\tau ) =  \left. \frac{E_n^k(q)}{(q-1)^g}\right|_{q=1}. $$
It is clear that from (\ref{hookploy}) that $H_{\lambda'}(q^d)$ is divisible by $(q-1)$ for all partitions $\lambda$ with $|\lambda| \geq 1$ and is divisible by $(q-1)^2$ unless $\lambda = (1)$. Comparing with Theorem \ref{conncompformthm} we see that if $n$ is odd, and $g \geq 2$ then
\begin{eqnarray*} 
\left. \frac{E_n^k(q)}{(q-1)^g}\right|_{q=1} &=&  \left. \frac{1}{2^r}  (q-1)(-q^{\frac{1}{2}})^{n^2(g-1)}  \frac{\mu(n)}{n}  \left( \sa^+_{(1)}   + \sa^-_{(1)} \right)^{r-k} \left(  \sa^+_{(1)}  - \sa^-_{(1)}\right)^k  \frac{\mathcal{H}_{(1)}^{g-1}(q^n)}{(q-1)^g} \right|_{q=1}\\
&=&   \mu(n) n^{g-2} 
\end{eqnarray*}
whereas if $n$ is even then
\begin{eqnarray*} 
\left. \frac{E_n^k(q)}{(q-1)^g}\right|_{q=1} &=& 0. \end{eqnarray*}

\end{proof}

\begin{eg}
We have
\begin{eqnarray*}
E_1^k(q) &=& (q-1) |G_1|^{g-1} = (q-1)^g.
\end{eqnarray*}

\begin{eqnarray*}
 \frac{E_2^k(q)}{(q-1)^g}  &=& \frac{1}{2^r} \left( \frac{|G_2|}{q-1} \right)^{g-1} \left(2^r + \frac{(3+1)^{r-k}(3-1)^k}{q^{g-1}} - \frac{4^r}{2(q+1)^{g-1}} \right)\\
 &=&  (q^3-q)^{g-1}+ 2^{r-k}(q^2-1)^{g-1} - 2^{r-1}(q^2-q)^{g-1}.
\end{eqnarray*} 

\begin{eqnarray*}
\frac{E_3^k(q)}{(q-1)^g} &=& (q^6-q^3)^{g-1}(q^2-1)^{g-1} +2^r(q^3-1)^{g-1}(q^2-1)^{g-1} \\
&&+2^r (q^5-q^2)^{g-1}(q+1)^{g-1} + \frac{4^r}{3}q^{3g-3}(q-1)^{2g-2} \\
&& - \frac{1}{3}(q^5+q^4+q^3)^{g-1} - 2^r(q^4-q^3)^{g-1}(q^2-1)^{g-1}-3^r(q^3-q^2)^{g-1}(q^2-1)^{g-1} .
\end{eqnarray*}
\end{eg}

Note that $E_n^k(q)$ is independent of $k$ when $n$ is odd. This is a consequence of the fact that the path components of $\mathcal{M}_{n}^\tau$ are pairwise isomorphic by Remark \ref{simpacm}.

\end{document}